\documentclass{lms}

\newtheorem{theorem}{Theorem}[section]         
\newtheorem{proposition}{Proposition}[section] 
\newtheorem{corollary}{Corollary}[section]
\newtheorem{lemma}{Lemma}[section]

\newnumbered{definition}{Definition}[section]
\newnumbered{remark}{Remark}
\newnumbered{note}{Note}
\newnumbered{example}{Example}
\newunnumbered{notation}{Notation} 

\def\eps{\varepsilon}

\def\v{\vec}

\def\a{\alpha}
\def\d{\delta}
\def\m{\times}

\def\F{\widehat}

\def\L{\Lambda}
\def\l{\lambda}

\def\k{\kappa}

\def\t{\tilde}

\def\o{\overline}

\providecommand{\ind}{\mathop{\rm ind}\nolimits}%

\def\Z{\mathbf{Z}}

\title{On a two--dimensional analog of Szemer\'{e}di's Theorem in Abelian groups}

\author{I.D. Shkredov }

\classno{35J25, 37A15}

\extraline{This work was supported by RFFI grant no. 06-01-00383,
President's of Russian Federation grant N 1726.2006.1 and INTAS
(grant no. 03--51--5-70).}

\begin{document}
\maketitle

\begin{abstract}
 Let $G$ be a finite Abelian group and
$A \subseteq G \m G$ be a set of cardinality at least $|G|^2/(\log
\log |G|)^c$, where $c>0$ is an absolute constant. We prove that
$A$ contains a triple $\{ (k,m), (k+d,m), (k,m+d) \}$, where
$d\neq 0$. This theorem is a two-dimensional generalization of
Szemer\'{e}di's theorem on arithmetic progressions.
\end{abstract}


\section{Introduction.}

    Szemer\'{e}di's theorem \cite{Sz} on arithmetic progressions states that
    an arbitrary set
    $A \subseteq \Z$ of positive density contains arithmetic progression of any length.
    This remarkable theorem has played a significant role in the development of two
    fields in mathematics : additive combinatorics (see e.g. \cite{Tao_Vu_book})
    and combinatorial ergodic theory (see e.g. \cite{Fu})
    A more precise statement of the theorem is as follows.

Let $N$ be a natural number. We set
\[
  a_k(N) = \frac{1}{N} \max \{ |A| ~:~ A \subseteq \{ 1,2, \dots, N \},
\]
\[
  A \mbox{ contains no arithmetic progressions of length  } k
   \},
\]
where $|A|$ denotes the cardinality of $A$.

\begin{theorem}[(Szemer\'{e}di, 1975)]
    For any $k\ge 3$ the following holds
\begin{equation}
  a_k(N) \to 0 \mbox{ as } N\to \infty \,.
\label{Sz_lab}
\end{equation}
\label{t:Szemeredi}
\end{theorem}
Clearly, this result implies van der Waerden's theorem \cite{Wdv}.

   In the simplest case $k=3$ of Theorem \ref{t:Szemeredi} was proven
 by K.F. Roth \cite{Rt} in 1953, who applied the Hardy -- Littlewood method to show that
\[
  a_3(N) \ll \frac{1}{\log \log N} \,.
\]
   At present, the best upper bound for $a_3(N)$
   is due to J. Bourgain \cite{Bu_new}.
He proved that
\begin{equation}
 a_3(N) \ll \frac{(\log \log N)^2}{ (\log N)^{2/3} } \,.
\end{equation}

Szemer\'{e}di's proof uses difficult combinatorial arguments. An
alternative proof was suggested by Furstenberg in \cite{Fu} (see
also \cite{Fu}). His approach uses the methods of ergodic theory.
Furstenberg showed that Szemer\'{e}di's theorem is equivalent to
the multiple recurrence of almost all points in any dynamical
system.

  A. Behrend  \cite{Be} obtained the following lower bound for  $a_3(N)$
\[
  a_3(N) \gg \exp (-C (\log N)^{\frac{1}{2}} ) \,,
\]
where $C$ is an absolute constant. A lower bound on $a_k(N)$ for
an arbitrary $k$ was given in \cite{Ra}.

Unfortunately, Szemer\'{e}di's methods give very weak upper
estimates for $a_k(N)$. The ergodic approach gives no estimates at
all.
  Only in  2001
W.T. Gowers \cite{Gow_m} obtained a quantitative result concerning
the rate at which $a_k(N)$ approaches zero for $k\ge 4$. He proved
the following theorem.

\begin{theorem}
For any $k\ge 4$, we have
$$
    a_k(N) \ll 1/ (\log \log N)^{c_k} \,,
$$
where the constant $c_k$ depends on $k$ only. \label{t:Gowers}
\end{theorem}

  In paper \cite{Sz2} and book \cite{Fu} the following problem was considered.
  Let $\{ 1,2, \dots, N \}^2$ be the two--dimensional lattice with basis  $\{(1,0)$, $(0,1)\}$.
Let also
\[
  L(N) = \frac{1}{N^2} \max \{~ |A| ~:~ A\subseteq \{ 1,2, \dots, N \}^2 ~\mbox{ and }~
\]
\[
  A \mbox { contains no triples of the form } \{ (k,m),~ (k+d,m),~ (k,m+d)  \}
\]
\begin{equation}\label{tri}
    \mbox { with positive } d \}.
\end{equation}
 A triple from (\ref{tri}) is called a
 {\it "corner".}
In \cite{Sz2,Fu} it was proven that  $L(N)$ tends to   $0$ as $N$
tends to infinity. W.T. Gowers (see \cite{Gow_m}) asked the
question of what is the rate of convergence of $L(N)$ to $0$.

The following theorem was proven in
\cite{Shkr_tri_LMS_DAN,Shkr_tri_LMS} (see also
\cite{Sol,Vu,Shkr_tri,Shkr_tri_DAN}).

\begin{theorem}
Let $\delta>0$, and $N\gg \exp \exp ( \d^{-73} )$. Let also $A$ be
a subset of $\{1,\dots,N\}^2$ of cardinality at least  $\delta
N^2$. Then $A$ contains a corner. \label{main_th-}
\end{theorem}

Thus, we have the estimate $L(N) \ll 1/ (\log \log N)^{1/73}$.

The question on upper estimates for $L(N)$ in the group
$\mathbf{F}_3^n$ was considered in \cite{Green_BCC} and
\cite{Lacey}.

A natural generalization of Theorem \ref{main_th-} above is
replace $\{1,\dots,N\}$ or $\Z/N\Z$ to an arbitrary Abelian group.
Such generalizations of Roth's theorem and Theorems
\ref{t:Szemeredi}, \ref{t:Gowers} were obtained in papers
\cite{Brown_Buhler,FGR_AP,Meshulam,Lev_AP,GT_U3}.

The main result of this paper is the following theorem.

Let $G$ be a finite Abelian group with additive group operation
$+$. In the case any triple of the form $\{ (k,m), (k+d,m),
(k,m+d) \}$, where $d\neq 0$ is called a corner.

\begin{theorem}
Let $G$ be a finite Abelian group and $A \subseteq G \m G$ be a
set of cardinality at least $|G|^2/(\log \log |G|)^c$, where $c>0$
is an absolute constant. Then $A$ contains a corner.
\label{main_th}
\end{theorem}

\begin{note*}
The constant $c$ in Theorem \ref{main_th} might  be taken as
$1/22$.
\end{note*}


  The proof of Theorem \ref{main_th} is contained in   \S 3,4,5,6
  and  proceeds by an iteration scheme
  as in all known effective proofs of Szemer\'{e}di--type theorems.

Let $G$ be an Abelian group, $A\subseteq G \m G$, $|A| \gg
|G|^2/(\log \log |G|)^c$ and we want to find a corner in $A$. At
each step of  our procedure we prove the following :  either $A$
is "sufficiently regular" or its "density" can be increased. A
suitable definition of  "sufficiently regular"\, sets (so--called
uniform sets) is one of the main aims of our proof.

If $A$ is a random set and $A$ has cardinality $\delta |G|^2$ then
it is easy to see that $A$ contains approximately $\d^3 N^3$
corners. We shall say $A$ is regular (or in other words {\it
$\a$--uniform}) if $A$ contains the same approximate number of
corners.

Let $E_1, E_2$ be subsets of $\L$,  where $\L \subseteq G$ to be
chosen later. Let $A$ be a subset of $E_1 \m E_2$ of cardinality
$\d |E_1| |E_2|$. We shall say that $A$ is {\it rectilinearly
$\a$--uniform} if, roughly speaking, the number of quadruples $\{
(x,y), (x+d,y), (x,y+s), (x+d,y+s) \}$ in $A^4$ is at most $(\d^4
+ \a) |E_1|^2 |E_2|^2$, $\a>0$ (in fact we need a slightly
different definition of $\a$--uniformity, which depends on the set
$\L$). In \S 3 we prove that if $E_1$, $E_2$ has small Fourier
coefficients and $A$ is rectilinearly $\a$--uniform then $A$ has
about the expected number of corners. Simple observation shows
(see e.g. \cite{Shkr_tri_LMS}) that the notion of rectilinearly
$\a$--uniformity cannot be expressed in terms of Fourier
transform, more precisely, there is a set, say $A_0$, with really
small Fourier coefficients but large number of quadruples $\{
(x,y), (x+d,y), (x,y+s), (x+d,y+s) \} \in A_0^4$. On the other
hand, we can define a rectilinearly $\a$--uniform set using
so--called rectilinear norm (see \S 3).

Suppose that $A$ fails to be rectilinearly $\a$--uniform. Roughly
speaking, it means that $A$ has no random properties. The last
observation can be expressed precisely by showing that $A$ has
increased density $\d + c(\d)$, $c(\d) > 0$ on some product set
$F_1 \m F_2$, $F_1 \subseteq E_1$, $F_2 \subseteq E_2$ (see \S 4).
Clearly, this density increment can only occur finitely many
times, because the density of any set does not exceed one. Thus,
our iteration scheme must stop after finite number of steps. It
means that we find a rectilinearly $\a$--uniform subset of the set
$A$ and consequently a corner in $A$.

 Unfortunately, the structure of $F_1 \m F_2$ need not be regular
 and we cannot make the next step of our procedure directly.
 To make $F_1 \m F_2$ regular, we pass to a subset of $\L$, say, $\L'$ and a vector $\v{t}=(t_1,t_2) \in G\m G$ such that
 $(F_1 - t_1) \cap \L'$, $(F_2 - t_2) \cap \L'$ has small Fourier coefficients.


 We are now in the situation we started with, but $A$ has a larger density
 and we iterate the procedure.
 This also can only occur finitely many times.
 In \S 6 we combine the arguments from the earlier sections and show that they give the bound
 that we stated in Theorem \ref{main_th}.

  In our prove we chose $\L$ to be a {\it Bohr set} (see \cite{Bu,Bu_new,Green_Sz_Ab} and others).
  Note that the best upper bound for $a_3 (N)$ was proven by J. Bourgain in \cite{Bu_new} using
  exactly
  these very sets.
  The properties of Bohr sets will be considered in \S 2.

  The constructions which we use develop the approach
  of \cite{Bu,Gow_m,Shkr_tri,Shkr_tri_LMS}.
  We improve our constant $c$ by more accurate calculations than in \cite{Shkr_tri_LMS}.

  In our forthcoming papers we are going to obtain a multidimensional
  analog of Theorem \ref{main_th}.

\begin{acknowledgements}
  The author is grateful to Professor N.G. Moshchevitin
  and
  Professor Ben Green.
\end{acknowledgements}


\section{ On Bohr sets.}

Let $G=(G,+)$ be a finite Abelian group with additive group
operation $+$. Suppose that $A$ is a subset of $G$. It is very
convenient to write $A(x)$ for such a function. Thus $A(x)=1$ if
$x\in A$ and $A(x)=0$ otherwise. By $\F{G}$ denote the Pontryagin
dual of $G$, in other words the space of homomorphisms $\xi$ from
$G$ to ${\bf T}$, $\xi : x \to \xi \cdot x$. It is well known that
$\F{G}$ is an additive group which is isomorphic to $G$. Also
denote by $N$ the cardinality of $G$.

One of the crucial moments in \cite{Bu} was the notion of Bohr
set.

  Let $S$ be a subset of $\F{G}$, $|S| = d$, $\eps>0$ be a real number.

\begin{definition}
Define the Bohr set  $\Lambda = \Lambda (S, \eps)$ by
\[
  \Lambda (S, \eps) = \{ n \in G ~|~ \| \xi \cdot n\| < \eps \mbox{ for all } \xi \in S \} \,.
\]
\end{definition}

We shall say that the set $S\subseteq \F{G}$ is  {\it generative
set} of Bohr set $\L$. The number  $d$~ is called  {\it dimension
} of Bohr set $\L$ and is denoted by $\dim \L$. If $M = \L + n$,
$n\in G$ is a translation of  $\L$, then, by definition, put $\dim
M = \dim \L$.

Another construction of Bohr set (so--called {\it smoothed } Bohr
set) was given in \cite{Tao_lect} and  \cite{Green_Sz_Ab}.

\begin{definition}
Let $0<\kappa<1$ be a real number. A Bohr set $\Lambda = \Lambda
(S, \eps)$ is called {\it regular}, if for an arbitrary $\eps'$
such that
\[
   | \eps - \eps'|  < \frac{\k}{100 d} \eps
\]
we have
\[
  1-\k < \frac{|\L (S,\eps')| }{ |\L (S,\eps) | } < 1+ \k \,.
\]
\end{definition}

We need  several results concerning Bohr sets (see \cite{Bu} and
\cite{Green_Sz_Ab}).

\begin{lemma}
Let $\L (S,\eps)$ be a Bohr set, $|S| = d$. Then
\[
  |\L (S,\eps)| \ge \eps^d N \,.
\]
\label{l:Bohr_est}
\end{lemma}
\begin{lemma}
Let $0 < \k < 1$ be a real number, and $\L (S,\eps)$ be a Bohr
set. Then there exists $\eps_1$ such that $
  \frac{\eps}{2} < \eps_1 < \eps
$ and such that  $\L (S,\eps_1)$ is a regular Bohr set.
\label{l:Reg_B}
\end{lemma}

{\it All Bohr sets will be regular in the article.}

\begin{definition}
Let $f,g$ be functions from $G$ to $\bf{C}$. By $f * g$ define the
function
\[
  ( f*g )(n) = \sum_{s\in G} f(s) g(n-s) \,.
\]
\end{definition}

\begin{definition}
Let $\eps \in (0,1]$ be a real number, and $\L (S,\eps_0)$ be a
Bohr set, $S \subseteq \F{G}$, $|S| = d$. A regular Bohr set $\L'
= \L (S', \eps')$ is called an {\it $\eps$--attendant } of $\L$ if
$S \subseteq S'$ and $\eps \eps_0 /2 \le \eps' \le \eps \eps_0$.
\end{definition}
Lemma \ref{l:Reg_B} implies that for an arbitrary Bohr set there
exists its $\eps$--attendant.

We shall consider that $S'=S$ unless stated otherwise.

Let $n$ be an arbitrary element of the group $G$, and $\L$ be a
Bohr set. We shall say that a Bohr set $\L'$ is an
$\eps$--attendant of $\L+n$, if $\L'$ is an $\eps$--attendant of
$\L$.

The following lemma is also due to J. Bourgain \cite{Bu}. We give
his proof for the sake of completeness.

\begin{lemma}
Let $\kappa>0$ be a real number, $S \subseteq \F{G}$, $\L = \L (S,
\eps)$ be a regular Bohr set, and $\L' = \L (S, \eps')$ its
$\kappa / (100 d)$--attendant. Then the number of $n's$ such that
$( \L * \L' ) (n) > 0$ does not exceed  $|\L|(1+\kappa)$, the
number of $n's$ such that $(\L * \L') (n) = |\L'|$ is greater than
$|\L|(1-\kappa)$ and
\begin{equation}\label{2k}
  \Big\| \frac{1}{|\L'|} (\L * \L')(n) - \L(n) \Big\|_1 < 2\kappa |\L| \,.
\end{equation}
\label{l:L_pm}
\end{lemma}

\begin{proof*}
If $(\L * \L^{'}) (n) > 0$ then there exists $m$ such that for any
$\xi \in S$, we have
\begin{equation}\label{tM2}
    \| \xi \cdot m \| < \frac{\k}{100d} \eps , \quad  \| \xi \cdot (n-m) \| < \eps \,.
\end{equation}
Using (\ref{tM2}), we get for all $\xi \in S$
\begin{equation}\label{}
  \| \xi \cdot n \|
             < \Big( 1 + \frac{\k}{100d} \Big) \eps \,,
\end{equation}
for all $\xi \in S$. It follows that
\begin{equation}\label{L+}
  n\in \L^{+} := \L \left( S, \left( 1 + \frac{\k}{100d} \right) \eps \right)  \,.
\end{equation}
By Lemma \ref{l:Reg_B} we have $|\L^{+}| \le (1+\k) |\L|$.

On the other hand, if
\begin{equation}\label{L-}
  n\in \L^{-} := \L \left( S, \left( 1 - \frac{\k}{100d} \right) \eps \right)
\end{equation}
then $(\L * \L^{'}) (n) = |\L'|$. Using Lemma \ref{l:Reg_B}, we
obtain $|\L^{-}| \ge (1-\k) |\L|$.

Let us prove (\ref{2k}). We have
\[
  \Big\| \frac{1}{|\L'|} (\L * \L')(n) - \L(n) \Big\|_1
    =
  \Big\| \frac{1}{|\L'|} (\L * \L')(n) - \L(n) \Big\|_{ l^1 ( \L^{+} \setminus \L^{-}) }
\]
\[
  \le
  |\L^{+}| - |\L^{-}| < 2 \k |\L|
\]
as required.
\end{proof*}

\begin{corollary*}
Lemma  \ref{l:L_pm} implies that  $|\L| \le |\L + \L'| \le (1+2\k)
|\L|$.
\end{corollary*}

\begin{note*} Let $\L^x (n) = \L (n-x)$.
Since  $(\L^x * \L') (n) = (\L * \L')(n-x)$, it follows that
(\ref{2k}) takes place for translations   $\L + x$.
\end{note*}

\begin{definition}
By $\L^{+}$ and $\L^{-}$ denote the Bohr sets defined in
(\ref{L+}) and (\ref{L-}), respectively, $\L^{-} \subseteq \L
\subseteq \L^{+}$.
\end{definition}

By Lemma \ref{l:L_pm} we have $|\L^{+}| \le |\L|(1+\kappa)$ and
$|\L^{-}| \ge |\L|(1-\kappa)$. Note that for any $s\in \L^{'}$, we
get $\L^{-} \subseteq \L + s$.

Suppose  $\L \subseteq G$ is  a Bohr set, and $\v{x} = (x_1, x_2)$
belongs to $G\m G$. By $\L + \v{x}$ denote the set $(\L+x_1) \m
(\L+x_2) \subseteq G\m G$. Let $\v{n} \in G\m G$. Let $\L(\v{n})$
denote the characteristic function of $\L \m \L$. We shall write
$\v{s} \in \L$, $\v{s} = (s_1, s_2)$, if $s_1 \in \L$ and $s_2 \in
\L$.

\begin{lemma}
  Suppose $\L$ is a Bohr set, $\L^{'}$ is its $\eps$--attendant, $\eps = \k/ (100 d)$, $\v{x}$ is a vector,
  and $E \subseteq G\m G$.
  Then
  \begin{equation}\label{}
    \Big| \d_{\L+ \v{x}} (E) -  \frac{1}{|\L|^2} \sum_{\v{n} \in \L+ \v{x}} \d_{\L^{'} + \v{n}} (E)  \Big|
        \le 4\k \,.
  \end{equation}
\label{l:smooth_d}
\end{lemma}

\begin{proof*}
We have
\[
  \sigma =
  \frac{1}{|\L|^2} \sum_{\v{n} \in \L+ \v{x}} \d_{\L^{'} + \v{n}} (E)
    =
  \frac{1}{|\L|^2 |\L^{'}|^2} \sum_{\v{s}} E(\v{s}) \sum_{\v{n}} \L(\v{n} - \v{x}) \L^{'} (\v{s} - \v{n})
\]
\[
    =
  \frac{1}{|\L|^2 |\L^{'}|^2} \sum_{\v{s}} E(\v{s}) \sum_{\v{n}} \L(\v{n} ) \L^{'} (\v{s} - \v{x} - \v{n})
\]
Using Lemma \ref{l:L_pm}, we get
\[
  \sigma = \frac{1}{|\L|^2} \sum_{\v{s}} E(\v{s}) \L( \v{s} - \v{x} ) + 4 \vartheta \k
         = \d_{\L+ \v{x}} (E) + 4 \vartheta \k \,,
\]
where  $|\vartheta | \le 1$. This completes the proof.
\end{proof*}

\begin{note*}
Clearly, the one--dimension analog of Lemma \ref{l:smooth_d} takes
place.
\end{note*}

Let $\L_1 = \L (S_1,\eps_1)$, $\L_2 = \L (S_2,\eps_2)$ be two Bohr
sets, $S_1, S_2 \subseteq \F{G}$. We shall write $\L_1 \le \L_2$,
if $S_1 \subseteq S_2$ and $\eps_1 \le \eps_2$.



\section{On $\alpha$--uniformity.}

%
%
%
%

Let $f$ be a function from $G$ to  $\mathbf{C}$, $N = |G|$. By
$\F{f}(\xi)$ denote the Fourier transformation of $f$
\[
  \F{f}(\xi) =  \sum_{x \in G} f(x) e( -\xi \cdot x) \,,
\]
where $e(x) = e^{2\pi i x}$. We shall use the following basic
facts
\begin{equation}\label{F_Par}
    \sum_{x\in G} |f(x)|^2 = \frac{1}{N} \sum_{\xi \in \F{G}} |\widehat{f} (\xi)|^2 \,.
\end{equation}
\begin{equation}\label{F_Par_sc}
    \sum_{x\in G} f(x) \overline{g(x)} = \frac{1}{N} \sum_{\xi \in \F{G}} \F{f}(\xi) \overline{\F{g}(\xi)} \,.
\end{equation}
\begin{equation}\label{svertka}
    \sum_{y\in G} |\sum_{x\in G} f(x) g(y-x) |^2
        = \frac{1}{N} \sum_{\xi \in \F{G}} |\widehat{f} (\xi)|^2 |\widehat{g} (\xi)|^2 \,.
\end{equation}

  Let $\Lambda$ be a Bohr  set,  and $A$ be an arbitrary subset of $\Lambda$.
  Let $|A| = \d |\Lambda|$.
  Define the {\it balanced} function of $A$ to be $f(s) = ( A(s) - \d) \L (s) = A(s) - \d \L (s)$.
\\
Let ${\bf D}$ denote the closed disk of radius $1$ centered at $0$
in the complex plane. Let $R$ be an arbitrary set. We write  $f: R
\to {\bf D}$ if $f$ is zero outside $R$.
\\
The following definition is due to Gowers \cite{Gow_m}.

\begin{definition}
  A function $f : \L \to {\bf D}$ is called {\it $\a$--uniform} if
\begin{equation}
 \| \F{f} \|_{\infty}  \le \a |\Lambda| \,.
\label{a_d0}
\end{equation}
\label{d:basic}
\end{definition}

We say that  $A$ is $\a$--uniform if its balanced function  is.

We shall write $\sum_s$ instead of $\sum_{s\in G}$ and
$\sum_{\xi}$ instead of $\sum_{\xi\in \F{G}}$.

Let us prove an analog of Lemma 2.2 from \cite{Gow_m}.

\begin{lemma}
Let $\L$ be a Bohr set, and let  $f : \L \to {\bf D}$ be an
$\a$--uniform function. Then we have
\[
   \sum_k |\sum_s f(s) g(k-s) |^2
     \le \a^2 |\L|^2 \| g \|^2_2 \,,
\]
for an arbitrary function $g$, $g: G \to {\bf D}$. \label{aux_l}
\end{lemma}

\begin{proof*}
By  (\ref{svertka}) we get
\begin{equation}\label{s:1}
  \sum_k |\sum_s f(s) g(k-s) |^2 =
  \sum_\xi |\widehat{f} (\xi)|^2 |\widehat{g} (\xi)|^2 \,.
\end{equation}
Since the function $f$ is $\a$--uniform, it follows that $\| \F{f}
\|_{\infty}  \le \a |\Lambda|$. Using this inequality and
(\ref{F_Par_sc}), we have
\begin{equation}\label{}
   \sum_k |\sum_s f(s) g(k-s) |^2 \le \a^2 |\L|^2 \frac{1}{N} \sum_\xi |\widehat{g} (\xi)|^2
   = \a^2 |\L|^2 \| g \|^2_2 \,.
\end{equation}
This completes the proof.
\end{proof*}

\begin{corollary}
Let $S \subseteq G$  be a set, and $\L^{'}$ be a Bohr set. Suppose
$E\subseteq \L^{'}$ is $\a$--uniform, and $E$ have the cardinality
$\d |\L^{'}|$. Let  $g$ be a function from $S$ to $[-1,1]$. Then
for all but $\a^{2/3} |S|$ choices of $k$ we have
\[
   \Big| (E * g) (k) - \d ( \L^{'} * g)(k) \Big| \le \a^{2/3} |\L^{'}| \,.
\]
\label{c:sv}
\end{corollary}

\begin{proof*}
Let  $f$ be the balanced function of $E\cap \L^{'}$. Using Lemma
\ref{aux_l}, we get
\begin{equation}
     \sum_k | ( E * g) (k) - \d ( \L^{'} * g)(k) |^2 =
     \sum_k |\sum_s f(s) g(k-s) |^2 \le
\end{equation}
\begin{equation}\label{}
     \le \a^2 |\L^{'}|^2 \| g \|^2_2 \le \a^2 |\L^{'}|^2 |S| \,.
\end{equation}
This concludes the proof.
\end{proof*}


  Let $\L_1$ and $\L_2$ be Bohr sets,
  and  $E_1\m E_2$ be  a subset of $\L_1 \m \L_2$.
Suppose  $f:\L_1\m \L_2 \to {\bf D}$ is a function.

\begin{definition}
Let $\alpha$ be a real number, $\a \in [0,1]$. A function $f : E_1
\m E_2\to {\bf D}$ is called {\it rectilinearly  $\a$--uniform} if
\begin{equation}
  \sum_{x,x',y,y'} f(x,y) \overline{f(x',y)}
  \overline{f(x,y')} f(x',y') \le \a |E_1|^2 |E_2|^2.
\label{a_d1}
\end{equation}
\label{a_d1_ref}
\end{definition}
Note that the function $f$ is $\a$--uniform iff
\begin{equation}
  \sum_{m,p} |\sum_k f(k,m) \overline{f(k,p)}|^2 \le \a |E_1|^2 |E_2|^2.
\label{a_d2}
\end{equation}

  Let $A$ be a subset of  $E_1\m E_2$, $|A| = \d |E_1| |E_2|$.
Define the {\it balanced} function of $A$ to be $f(x,y)=( A(x,y)-
\delta ) \cdot (E_1\m E_2) (x,y)$.
  We say that $A\subseteq E_1\m E_2$ is rectilinearly $\a$--uniform if its balanced function is.

  Let $f$ be an arbitrary function, $f : G \m G \to \mathbf{C}$.
Define $\| f \|$ by the formula
\begin{equation}
  \| f \| =
  \Big| \sum_{x,x',y,y'} f(x,y) \overline{f(x',y)}
  \overline{f(x,y')} f(x',y') \Big|^{\frac{1}{4}}
\label{norm}
\end{equation}

\begin{lemma}
$\| \cdot \|$ is a norm. \label{l:norm}
\end{lemma}
\begin{proof*}
See \cite{Shkr_tri}.
\end{proof*}

\begin{definition}
Let $\L$ be a Bohr set, $Q \subseteq \L$, $|Q| = \d |\L|$,
$\a,\eps$ are positive numbers. A set $Q$ is called {\it
$(\a,\eps)$--uniform} if there exists $\L'$ such that $\L'$ is an
$\eps$--attendant set of $\L$ and the set
\[
   B := \{ m \in \L  ~|~  \| (Q \cap (\L'+m) - \d (\L'+m) ) \F{} ~ \|_{\infty}  \ge \a |\L'| \}
\]
has the cardinality at most $\a |\L|$
\begin{equation}\label{c:sm_B}
  |B| \le \a |\L| \,,
\end{equation}
further
\begin{equation}\label{c:sm_sq}
  \frac{1}{|\L|} \sum_{m\in \L} |\d_{\L'+m} (Q) - \d |^2 \le \a^2 \,.
\end{equation}
and
\begin{equation}\label{c:sm_total_F}
  \| ( Q \cap \L - \d \L ) \F{} ~ \|_{\infty} \le \a |\L| \,.
\end{equation}
\end{definition}

Certainly, this definition depends on $\L$. and $\L'$. We do not
assume that $\L'$ has the same generative set as $\L$. If $Q$ is
$(\a,\eps)$--uniform and $\L'$ is an $\eps$--attendant set of $\L$
then we shall mean sometimes that $\L'$ is an $\eps$--attendant
set of $\L$ such that (\ref{c:sm_B}) --- (\ref{c:sm_total_F})
hold.

\begin{note*}
Let
\[
  B^* = \{ m \in \L  ~|~  | \d_{\L'+m} (Q) - \d |  \ge \a^{2/3}  \} \,.
\]
Condition (\ref{c:sm_sq}) implies  that $|B^*| \le \a^{2/3} |\L|$.
\end{note*}

\begin{note*}
Condition (\ref{c:sm_total_F}) is not so important as
(\ref{c:sm_B}) and (\ref{c:sm_sq}). The inequality
\[
  \| ( Q \cap \L - \d \L ) \F{} ~ \|_{\infty} \le 4 \a |\L|
\]
follows from (\ref{c:sm_B}), (\ref{c:sm_sq}) (see Proposition
\ref{l:l2}).
\end{note*}

Let $\L_1$, $\L_2$ be Bohr sets, $\L_1 \le \L_2$, $\eps>0$ be a
real number. Let also $E_1$, $E_2$ be subsets of $\L_1$, $\L_2$,
respectively,  and $|E_1| = \beta_1 |\L_1|$, $E_2 = \beta_2
|\L_2|$.

\begin{definition}
A function $f : E_1 \m E_2 \to {\bf D} $ is called {\it
rectilinearly $(\a, \eps)$--uniform} if there exists $\L'$ such
that $\L'$ is an $\eps$--attendant of $\L_1$ and
\[
  \| f \|_{\L_1\m \L_2, \eps}^4 = \sum_{i\in \L_1} \sum_{j\in \L_2}
                                  \sum_k \sum_{m,u}
                                  \L' (m-k-i) \L' (u-k-i) ~\m
\]
\begin{equation}\label{f:f_for_norm}
                                  |\sum_r \L' (k+r-j) f(r,m) f(r,u) |^2
                                  \le
                                  \a \beta_1^2 \beta_2^2 |\L'|^4 |\L_1|^2 |\L_2| \,.
\end{equation}
\end{definition}



Let $\L_1$, $\L_2$ be Bohr sets, $\L_1 \le \L_2$. Let also $E_1$,
$E_2$ be subsets of $\L_1$, $\L_2$, respectively,  and $|E_1| =
\beta_1 |\L_1|$, $E_2 = \beta_2 |\L_2|$.

\begin{definition}
Let  $A \subseteq E_1 \m E_2$, $|A| = \d \beta_1 \beta_2 |\L_1|
|\L_2|$, and $f(\v{s}) = A(\v{s}) - \delta (E_1 \m E_2) (\v{s})$.
$A$ is called {\it rectilinearly $(\a, \a_1, \eps)$--uniform} if
there exist $\L'$, $\L'_\eps$ such that $\L'$ is an
$\eps$--attendant of $\L_1$, $\L'_\eps$ is an $\eps$--attendant of
$\L'$ and the set
\[
  B = \{ l \in \L_1 ~|~ \| f_l \|^4_{\L' \m \L_2,\eps} >
         \a \beta_1^2 \beta_2^2 |\L'_\eps|^4 |\L'|^2 |\L_2| \} \,,
\]
where $f_l (\v{s}) := f(s_1 + l, s_2) \L'(s_1) $, $l\in \L_1$ has
the cardinality at most $\a_1 |\L_1|$, \label{def:rectangle}
\end{definition}

Note that
\[
  \| f_l \|^4_{\L' \m \L_2,\eps} = \sum_{i\in \L'} \sum_{j\in \L_2}
                                   \sum_k \sum_{m,u}
                                   \L'' (m-k-i) \L'' (u-k-i)
                                   |\sum_r \L'' (k+r-j) f_l (r,m) f_l(r,u) |^2
\]
\[
                                  =
                                   \sum_{i\in \L' + l} \sum_{j\in \L_2}
                                   \sum_k \sum_{m,u}
                                   \L'' (m-k-i) \L'' (u-k-i) ~\m
                                   |\sum_r \L'' (k+r-j) \t{f}_l (r,m) \t{f}_l(r,u) |^2 \,,
\]
where $\L'' = \L'_\eps$ and $\t{f}$ is a restriction of  $f$ to
$(\L'+l) \m \L_2$.

\begin{note*}
We need parameter $\a_1$ to decrease the constant $c$ in Theorem
\ref{main_th}. To obtain Theorem \ref{main_th} with $c$ equals,
say, $1000$, one can put  $\a_1 = \a$.
\end{note*}

\begin{lemma}
  Let $\L$ be a Bohr set.
  Suppose $\L'$ is an $\eps$--attendant of $\L$,
  $\L''$ is an $\eps$--attendant $\L'$  and an $\eps^2$--attendant of $\L$,
  $\eps = \a^2 / 4 (100 d)$,
  $Q \subseteq \L$, $|Q| = \d \L$, and $\a>0$.
  Let
  \[
    \Omega_1 = \{ s\in \L ~|~ | \d_{\L'+s} (Q) - \d | \ge 4 \a^{1/2}
    \mbox{ or }~ \frac{1}{|\L'|} \sum_{n\in \L'+s} | \d_{\L'' + n} (Q) - \d |^2 \ge 4 \a^{1/2} \} \,.
  \]
  \[
    \Omega_2 = \{ s\in \L ~|~ \| (Q \cap (\L' + s)  - \d (\L'+s) ) \F{} ~ \|_{\infty} \ge 4 \a^{1/4} |\L'| \}  \,.
  \]
1)
  If
  \begin{equation}\label{c:T1}
    \frac{1}{|\L|} \sum_{n\in \L} | \d_{\L'' + n} (Q) - \d |^2 \le \a^2 \,,
  \end{equation}
  then   $ |\Omega_1| \le 4 \a^{1/2} |\L|$.
  \\
2)
  If
  \begin{equation}\label{c:T2}
    \Omega^* = \{ s\in \L ~|~ \| (Q \cap (\L'' + s) - \d (\L''+s) ) \F{} ~ \|_{\infty} \ge \a |\L''| \}
  \end{equation}
  has the cardinality at most $\a |\L|$,
  then   $ |\Omega_2| \le 4 \a^{1/2} |\L|$.
  \\
3)
  Suppose $Q$ is $(\a,\eps^2)$--uniform subset of $\L$
  and $\L''$ is an $\eps^2$--attendant of $\L$ such that (\ref{c:sm_B}) --- (\ref{c:sm_total_F}) hold.
  Let
  \[
    \t{\Omega} = \{ s\in \L ~|~ \mbox{ Set } ( Q - s ) \cap \L' \mbox{ is not }
                                (8 \a^{1/4},\eps)\mbox{--uniform } \} \,.
  \]
  Then $|\t{\Omega}| \le 8 \a^{1/2} |\L|$.
\label{l:intermediate}
\end{lemma}

\begin{proof*}
Let us prove $1)$. Let $\d'_n = \d_{\L' + n} (Q)$, $\d''_n =
\d_{\L'' + n} (Q)$, $\k = \a^2 / 4$, and $\epsilon = \a^{1/2}$.
Consider the sets
\[
  B_s = \{  n \in \L' + s ~|~ | \d''_n - \d | \ge \epsilon \} \,, \quad
  G_s = \{  n \in \L' + s ~|~ | \d''_n - \d | < \epsilon \} \,, s\in \L
\]
and the sets
\[
  B = \{ s \in \L ~|~ |B_s| \ge \epsilon |\L'| \} \,, \quad
  G = \{ s \in \L ~|~ |B_s| < \epsilon |\L'| \} \,.
\]
If $s\in G$ then $|B_s| < \epsilon |\L'|$. Using Lemma
\ref{l:smooth_d}, we have
\[
  | \d'_s - \d | \le \Big| \frac{1}{|\L'|} \sum_{x\in \L'+s} \d''_x   -   \d \Big| + 4 \k
  \le
    \frac{1}{|\L'|}  \sum_{x\in \L'+s} |\d''_x   -   \d | + 4 \k
  \le
\]
\begin{equation}\label{tmp:17:17}
  \le
    \frac{1}{|\L'|}  \sum_{x\in B_s} |\d''_x   -   \d | + \frac{1}{|\L'|}  \sum_{x\in G_s} |\d''_x   -   \d | + 4 \k
  <
      \epsilon + \frac{\epsilon |G_s|}{|\L'|} + 4 \k
  \le
      4 \epsilon \,.
\end{equation}
Besides that for $s\in G$, we get
\begin{equation}\label{tmp:17:17+}
  \frac{1}{|\L'|} \sum_{x\in \L'+s} |\d''_x   -   \d |^2
    \le
        \frac{1}{|\L'|} \sum_{x\in B_s} |\d''_x   -   \d |^2
            +
        \frac{1}{|\L'|} \sum_{x\in G_s} |\d''_x   -   \d |^2
            \le
            \epsilon + \epsilon^2 \le 2 \epsilon \,.
\end{equation}
Let us estimate the cardinality of $B$. We have
\[
 \a^2 \ge
        \frac{1}{|\L|} \sum_{s\in B} |\d''_s - \d|^2
    \ge
        \frac{1}{|\L'| |\L|} \sum_{s\in B} \sum_{n \in \L'+s} |\d''_n - \d|^2 - 4\k
    \ge
\]
\[
    \ge
        \frac{1}{|\L'| |\L|} \sum_{s\in B} \sum_{n \in B_s} |\d''_n - \d|^2 - 4 \k
    \ge
        \frac{|B| \epsilon^3 |\L'|}{|\L'| |\L|}  - 4\k \,.
\]
It follows that, $|B| \le 4 \a^{1/2} |\L|$. Using
(\ref{tmp:17:17}), (\ref{tmp:17:17+}) we get $\Omega_1 \subseteq
B$ and  $1)$ is proven.
To prove $2)$ it suffices to note that                                                                  
\[
  \frac{1}{|\L||\L'|} \sum_{s\in \L} \| (Q \cap (\L'' + s)  - \d (\L''+s) ) \F{} ~  \|_{\infty}
      =
  \frac{1}{|\L||\L'|} \sum_{s\in \Omega^*} \| (Q \cap (\L'' + s) - \d (\L''+s) ) \F{} ~  \|_{\infty}
    +
\]
\[
    +
  \frac{1}{|\L||\L'|} \sum_{s\in ( \L \setminus \Omega^* ) } \| (Q \cap (\L'' + s) - \d (\L'' + s)) \F{} ~  \|_{\infty}
      \le
            \a + \frac{\a |\L'|}{|\L||\L'|} |\L \setminus \Omega^*|
      \le 2\a \,.
\]
and define the sets $B'_s$, $G'_s$, $B'$, $G'$ ~:~
\[
  B'_s = \{  n \in \L' + s ~|~ \| (Q \cap (\L'' + n) - \d (\L''+n) ) \F{} ~ \|_{\infty} \ge \epsilon_1 |\L''| \} \,,
\]
\[
  G'_s = \{  n \in \L' + s ~|~ \| (Q \cap (\L'' + n) - \d (\L''+n) ) \F{} ~ \|_{\infty} < \epsilon_1 |\L''| \}, \,~ s\in \L \,.
\]
\[
  B' = \{ s \in \L ~|~ |B_s| \ge \epsilon_1 |\L'| \}
  \quad \mbox{ and } \quad
  G' = \{ s \in \L ~|~ |B_s| < \epsilon_1 |\L'| \} \,,
\]
where $\epsilon_1 = \a^{1/4}$. After that we can apply the same
arguments as above, using Lemma \ref{l:L_pm} instead of Lemma
\ref{l:smooth_d}.

Let us prove $3)$. Since  $Q$ is $(\a,\eps^2)$--uniform subset of
$\L$, it follows that $Q$ satisfies (\ref{c:T1}).
Also we have $|\Omega^*| \le \a |\L|$, and $|B|, |B'| \le 4
\a^{1/2} |\L|$ (see above). It is easily shown that for all $s
\notin B \cup B'$ the set $( Q - s ) \cap \L'$ is $(8
\a^{1/4},\eps)$--uniform. This completes the proof.
\end{proof*}

In the same way we can prove

\begin{proposition}
  Let $\L$ be a Bohr set,  and $E\subseteq \L$, $|Q| = \d |\L|$ be  $(\a, \eps)$--uniform,
  $\eps = \a / 4 (100 d)$.
  Then
  \begin{equation}\label{}
   \| (Q \cap \L - \d \L) \F{} ~ \|_{\infty} < 4 \a |\L| \,.
  \end{equation}
\label{l:l2}
\end{proposition}
We will not, however, use this fact.

Let $\L_1$, $\L_2$ be Bohr sets, $\L_1 \le \L_2$, and $E_1
\subseteq \L_1$, $E_2 \subseteq \L_2$, $|E_1| = \beta_1 |\L_1|$,
$|E_2| = \beta_2 |\L_2|$. By  $\mathcal{P}$ denote the $E_1 \m
E_2$. Let $A\subseteq \mathcal{P}$, $|A| = \d |E_1| |E_2|$. Denote
by $H$ and $W$ two copies of the set $A$.

\begin{theorem}
Let $f : \mathcal{P} \to {\bf D}$  be a rectilinearly
$(\a,\eps)$--uniform function. Suppose that sets $E_1$, $E_2$ are
$(\a_0,\eps)$--uniform, $\a_0 = 2^{-50} \a^2 \beta_1^{12}
\beta_2^{12}$, $\eps = 2^{-10} \eps_0^2$, $\eps_0 = (2^{-10}
\a_0^2) /(100 d)$. Let $\L_1$ be an $\eps_0$--attendant of $\L_2$.
Then either
\begin{equation}\label{f:f_residual}
  | \sum_{s_1,s_2, r}
  H (s_1,s_2) W ( s_1 + r, s_2 + r ) f(s_1, s_2 + r) |
  \le 16 \a^{1/4} \d^{3/4} \beta_1^2 \beta_2^2 |\L_1|^2 |\L_2|
\end{equation}
or there exists a Bohr set $\L'$, two sets $F_1$, $F_2$ and a
vector $\v{y} = (y_1, y_2) \in G\m G$, $F_1 \subseteq E_1 \cap
(\L' + y_1) $, $F_2 \subseteq E_2 \cap (\L' + y_2) $ such that
$\L'$ is an $\eps_0$--attendant of $\L_1$ and
\begin{equation}\label{est::delta'}
  \quad |F_1| \ge 2^{-2} \beta_1 |\L'|,  \quad
        |F_2| \ge 2^{-2} \beta_2 |\L'| \,\, \mbox{ and }
\end{equation}
\begin{equation}\label{est::card'}
  \quad  \d_{F_1 \m F_2} (A) \ge 2\d \,.
\end{equation}
\label{t:tmp_tha}
\end{theorem}
\begin{proof*}
Let $\L'$ be an $\eps_0$--attendant of $\L_1$ to be chosen later.

Let
\[
   \Omega^{(1)}_1 = \{ s \in \L_1  ~|~ \| (E_1 \cap (\L'+s) - \d (\L'+s) ) \F{} ~ \|_{\infty} \ge \a_0  \} \,, \quad
\]
\[
   \Omega^{(1)}_2 = \{ s \in \L_1  ~|~ | \d_{\L'+s}( E_1) - \beta_1 | \ge \a_0^{2/3} \} \,,
\]
and
\[
  \Omega^{(2)}_1 = \{ s \in \L_2  ~|~ \| (E_2 \cap (\L'+s) - \d (\L'+s) ) \F{} ~ \|_{\infty} \ge \a_0  \} \,, \quad
\]
\[
  \Omega^{(2)}_2 = \{ s \in \L_2  ~|~ | \d_{\L'+s}( E_2) - \beta_2 | \ge \a_0^{2/3} \} \,,
\]
Let also $\Omega_1 = \Omega^{(1)}_1 \cup \Omega^{(1)}_2$, and
$\Omega_2 = \Omega^{(2)}_1 \cup \Omega^{(2)}_2$. By assumption
the sets $E_1$, $E_2$ are $(\a_0, \eps)$--uniform. Let $\L'$ be
$\eps_0$--attendant of $\L_1$ such that (\ref{c:sm_B}) ---
(\ref{c:sm_total_F}) hold. Using definitions and Lemma
\ref{l:intermediate}, we get
$|\Omega^{(1)}_l| \le \a_0^{2/3} |\L_1| $, $|\Omega^{(2)}_l| \le
\a_0^{2/3} |\L_2|$, $l=1,2$. Hence $|\Omega_1| \le 2 \a_0^{2/3}
|\L_1|$ and $|\Omega_2| \le 2 \a_0^{2/3} |\L_2|$.

Let $g_i(\v{s}) = g_i (k,m) = W (k,m) \L' (k-i) $, $i\in \L_1$,
and $h_j(\v{s}) = h_j (k,m) = H (k,m) \L' (m-j)$, $j\in \L_2$.
We have $k \in \L_1$, $m \in \L_2$ and $k+r \in \L_1$ in (\ref{f:f_residual}).                              
It follows that the sum (\ref{f:f_residual}) does not exceed
$|\L_1|^2 |\L_2|$. Let also $\l_i = \L' + i$, and $\mu_j = \L' +
j$. Using Lemma \ref{l:L_pm}, we get
\[
  \sigma_0 =
                \sum_{s_1,s_2, r } H (s_1,s_2) W ( s_1 + r, s_2 + r ) f(s_1, s_2 + r)
           =
\]
\[
           =
             \sum_{k,m} \sum_{r}
             H (k,m) W ( k + r, m+r ) f(k,m+r) \L_1(k+r) \L_2(m)
           =
\]
\[
             \frac{1}{|\L'|^2}
             \sum_{k,m} \sum_{r}
             H (k,m) W ( k + r, m+r ) f(k,m+r) (\L_1 * \L') (k+r) (\L_2 * \L') (m) + 16 \vartheta_0 \k |\L_1|^2 |\L_2|
           =
\]
\begin{equation}\label{f:begG}
             \frac{1}{|\L'|^2}
             \sum_{i\in \L_1} \sum_{j\in \L_2}
             \sum_{k,m} \sum_r
             h_j (k,m) g_i ( k + r, m+r ) f(k,m+r)
             + 16 \vartheta_0 \k |\L_1|^2 |\L_2| \,,
\end{equation}
where $ |\vartheta_0| \le 1$ and $\k \le 2^{-10} \a_0^2$. Split
the sum $\sigma_0$ as
\begin{equation}\label{f:4+R}
  \sigma_0 = \widetilde{\sigma}_0 + \sigma_0^{'} + \sigma_0^{''} + \sigma_0^{'''} + R\,,
\end{equation}
The sum $\widetilde{\sigma}_0$ is taken over $i \notin \Omega_1,
j\notin \Omega_2$, the sum $\sigma_{0}^{'}$ is taken over $i \in
\Omega_1, j\notin \Omega_2$, the sum $\sigma_{0}^{''}$ is taken
over $i \notin \Omega_1, j\in \Omega_2$, the sum
$\sigma_{0}^{'''}$ is taken over $i \in \Omega_1, j\in \Omega_2$
and $|R| \le 16 \eps |\L_1|^2 |\L_2|$. Let us estimate
$\sigma_0^{'}$, $\sigma_0^{''}$ and  $\sigma_0^{'''}$. Rewrite
$\sigma_0$ as
\begin{equation}\label{f:f_residual1}
  \sigma_0 = \frac{1}{|\L'|^2} \sum_{i \in \L_1} \sum_{j\in \L_2} \sum_{k,m} \sum_r h_j(k-r,m) g_i(k,m+r) f(k-r,m+r) + R\,.
\end{equation}
Let $i$ and $j$ in the sum (\ref{f:f_residual1}) be fixed. We have
$k\in \l_i$ and $m\in \mu_j$. Further if  $f(k-r,m+r)$ is not
zero, then $k-r \in \L_1$. It follows that $r\in \l_i - \L_1 = \L'
- \L_1 + i$. The set $\L'$ is $\eps_0$--attendant of $\L_1$. Using
Lemma \ref{l:L_pm}, we obtain that $r$ belongs to a set of
cardinality at most  $2 |\L_1|$. Hence
\begin{equation}\label{e:sigma'}
  |\sigma_0^{'}|  \le \frac{1}{|\L'|^2} 2 |\Omega_1| \cdot |\L_2| \cdot |\L'|^2 |\L_1|
                  \le 2 \a_0^{2/3} |\L_1|^2 |\L_2| \,.
\end{equation}
In the same way $|\sigma_0^{''}| \le 2 \a_0^{2/3} |\L_1|^2 |\L_2|
$ and $|\sigma_0^{'''}| \le 2 \a_0^{2/3} |\L_1|^2 |\L_2| $.

Take $i$ and $j$ such that   $i\notin \Omega_1$, $j\notin
\Omega_2$. Let  $g (\v{s}) = g_i (\v{s})$, $h (\v{s}) = h_j
(\v{s})$, and $\L_1 \m \mu_j = \L_1^{(1)} \m \L_2^{(1)}$, $\l_i \m
\L_2 = \L_1^{(2)} \m \L_2^{(2)}$. Let
$E_2^{(1)} = E_2 \cap \L_2^{(1)}$, $E_1^{(2)} = E_1 \cap
\L_1^{(2)}$, $\beta_2^{(1)} = |E_2^{(1)}| / |\L_2^{(1)}|$, and $
\beta_1^{(2)} = |E_1^{(2)}| / |\L_1^{(2)}|$. We have
\begin{equation}\label{}
 \sigma = \sigma_{i,j}
        = \sum_{s_1,s_2, r }  h (s_1,s_2) g ( s_1 + r, s_2 + r ) f(s_1, s_2 + r) =
 \end{equation}
\begin{equation}\label{tmp:16:10_11}
        = \sum_{k,m} h(k,m) E_2^{(1)} (m) \sum_{r} g(k+r,m+r) f(k,m+r)
\end{equation}
Note that $k$ in (\ref{tmp:16:10_11}) belongs to $\L^{(2)}_1$.
Using  the Cauchy--Schwartz inequality, we obtain
\begin{equation}\label{}
 |\sigma|^2 \le \| h \|_2^2 \sum_{k,m} E_2^{(1)} (m) |\sum_r g(k+r,m+r) f(k,m+r) |^2 =
\end{equation}
\[
 = \| h \|_2^2 \sum_{k,m} E_2^{(1)} (m) \sum_{r,p} g(k+r,m+r) f(k,m+r) g(k+p,m+p) f(k,m+p) =
\]
\[
 = \| h \|_2^2 \sum_{k,m,u} g(k,m) g(k+u,m+u) \sum_r E_2^{(1)} (m-r) f(k-r,m) f(k-r,m+u) =
\]
\[
 = \| h \|_2^2 \sum_{k,m,u} g(k,m) g(k+u,m+u) E_1^{(2)} (k) E_1^{(2)} (k+u)
\]
\[
    \cdot \sum_r E_2^{(1)} (m-r) f(k-r,m) f(k-r,m+u) \,.
\]
We have $k\in \L_1^{(2)}$ and $k-r \in \L_1$. It follows that
$r\in k-\L_1 \in \L_1^{(2)} - \L_1$. Since  $m-r \in \L_2^{(1)}$
it follows that $m\in \L_2^{(1)} + r \in \L_2^{(1)} + \L_1^{(2)} -
\L_1$. On the other hand $k+u \in \L_1^{(2)}$. Hence $u\in
\L_1^{(2)} - \L_1^{(2)}$ and $m+u \in \L_2^{(1)} + \L_1^{(2)} -
\L_1 + \L_1^{(2)} - \L_1^{(2)}$. Let $\tilde{\L}_i = \L' + \L' +
\L' + \L' + \L_1 + i$. Then $m, m+u \in \tilde{\L}_i + j = Q_{ij}
= Q$. Using Lemma \ref{l:L_pm} for the Bohr set $\L_1$ and its
$\eps_0$--attendant  $\L'$, we obtain that the cardinality of
$\tilde{\L}_i$ does not exceed $5|\L_1|$. Using the
Cauchy--Schwartz inequality, we get
\begin{equation}\label{}
  |\sigma|^4 \le \| h \|_2^4 \left( \sum_k \sum_{m,u}  g(k,m) g(k+u,m+u) \right)
\end{equation}
\[
                           \cdot
                           \Big( \sum_{k,m,u} E_2^{(1)} (k) E_1^{(2)} (k+u) \sum_{r,r'} E_1^{(2)} (m-r) E_1^{(2)} (m-r')
                                \m
\]
\[
                                  f(k-r,m) f(k-r,m+u) f(k-r',m) f(k-r',m+u) \Big) \,.
\]
Let $\sigma^* = \sigma_{ij}^* = \sum_k \sum_{m,u}  g(k,m)
g(k+u,m+u)$.
Let
\[
  \Omega^{'} = \{ s \in \L_2  ~|~ | \d_{\L_1+s}( E_2) - \beta_2 | \ge 4\a_0^{1/2} \mbox{ or }
\]
\[
                                    \frac{1}{|\L_1|} \sum_{n\in \L_1 +s} |\d_{\L'+n} (E_2) - \beta_2 |^2  \ge 4 \a_0^{1/2}  \},
  \mbox{ and }
  G^{'} = \L_2 \setminus \Omega^{'} \,.
\]
By assumption $\L_1$ is an $\eps_0$--attendant of $\L_2$ and $E_2$
is an $(\a_0,\eps)$--uniform subset of $\L_2$. Using Lemma
\ref{l:intermediate}, we get $|\Omega^{'}| \le 8 \a_0^{1/2}
|\L_2|$. Let $\tilde{\L} = \L' + \L' + \L' + \L' + \L_1$. Since
$\L'$ is an $\eps_0$--attendant of $\L_1$, it follows that for any
$s\in G^{'}$ we have $ | \d_{ \tilde{\L} + s } ( E_2) - \beta_2 |
< 8 \a_0^{1/2}$ and $ \sum_{n\in \tilde{\L} +s} |\d_{\L'+n} (E_2)
- \beta_2 |^2  < 8 \a_0^{1/2} |\t{\L}|$. For an arbitrary  $i\in
\L_1$ consider the set
\[
 \Omega^{*} = \Omega^{*}_i = \{~ j \in \L_2 ~|~ | \d_{ \tilde{\L}_i + j } ( E_2) - \beta_2 | \ge 8 \a_0^{1/2}
                                                    \mbox{ or }
\]
\begin{equation}\label{f:tilde}
                                    \frac{1}{|\t{\L}_i|} \sum_{n\in \t{\L}_i +j} |\d_{\L'+n} (E_2) - \beta_2 |^2  \ge 8\a_0^{1/2} ~\} \,.
\end{equation}
Since  $ (\L_2 \setminus \Omega^{*}_i ) \supseteq ( \L_2 \cap
(G^{'} - i) )$ it follows that $\Omega^{*}_i \subseteq ( \L_2
\setminus (G^{'} - i) )$. Since  $\L_1$ is an $\eps_0$--attendant
of $\L_2$, it follows that $| \L_2 \setminus (G^{'} - i) | = |
(\L_2 + i) \setminus G^{'} | \ge |\L_2^{-} \cap G^{'}| \ge (1 - 8
\a_0^{1/2} - 8\k_0) |\L_2|$, $\k_0 \le \a_0^2$. Hence
$|\Omega^{*}_i| \le 8 \a_0^{1/2} |\L_2| + 8 \k_0 |\L_2| \le 16
\a_0^{1/2} |\L_2|$.
This yields
\begin{equation}\label{est:*}
  \frac{1}{|\L'|^2}
    \sum_{i\notin \Omega_1, j \in \Omega^{*}_i} |\sigma_{ij}|
        \le
            \frac{1}{|\L'|^2}  \sum_{i\notin \Omega_1} ( 16 \a_0^{1/2} |\L_2|  2 |\L'|^2 |\L_1| )
        \le
            32 \a_0^{1/2} |\L_1|^2 |\L_2| \,.
\end{equation}
We have $j \notin \Omega_2$. Suppose in addition that  $j\notin
\Omega^{*}_i$. Let $\Omega'_2 = \Omega'_2 (i) = \Omega_2 \cup
\Omega^{*}_i$.

\begin{lemma}
    For any $i\notin \Omega_1$ and any $j\notin \Omega^{*}_i$ the following holds
    we have
    either
    \begin{equation}\label{f:sigma_star}
             |\sigma_{ij}^*| \le 16 \d \beta^2_1 \beta^2_2 |\L'|^{2} |\L_1|^2 |\L_2| \,.
    \end{equation}
    or
    there exist two sets $F_1$, $F_2$
    and a vector $\v{y} = (y_1, y_2) \in G\m G$,
    $F_1 \subseteq E_1 \cap (\t{\L} + y_1) $, $F_2 \subseteq E_2 \cap (\t{\L} + y_2) $ such that
    (\ref{est::delta'}) and (\ref{est::card'}) hold.
\label{l:w_est}
\end{lemma}
\begin{note*}
    Let $T$ be a subset of $G$, $|T| = \d |G|$,
    $E_1 = E_2 = G$, $\beta_1 = \beta_2 = 1$
    and let $g$ be the characteristic function
    of the set $\mathcal{A} = \bigsqcup_{x\in G} (\{ x \} \m \{ T+x \})$.
    Then it is easy to see that inequality (\ref{f:sigma_star})
    is best possible in the case (up to constants).
    On the other hand (\ref{est::delta'}), (\ref{est::card'}) does not hold
    with $A$ equals $\mathcal{A}$.
\end{note*}
\begin{proof}
Let $\t{E}_2^{(2)} = E_2 \cap Q$ and $\o{E}_2^{(2)} (x) =
\t{E}_2^{(2)} (-x)$. We have
\[
  \sigma^{*}_{ij}
                =
                    \sum_{k,m,u} g(k,m) g(k+u,m+u)
                        \le \sum_{k,m,u} g(k,m) E_1^{(2)} (k+u) \t{E}_2^{(2)} (m+u)
\]
\begin{equation}\label{2tmp_1}
               = \sum_{k,m}  g(k,m) ( E_1^{(2)} * \o{E}_2^{(2)} ) (k-m)
               = \sum_{k',m}  g(k'+m,m) ( E_1^{(2)} * \o{E}_2^{(2)} ) (k') \,.
\end{equation}
If $k'$ is fixed then the variable $m$ in (\ref{2tmp_1}) belongs
to the set of the cardinality $|\L'|$. Recall that  $|Q_{ij}| \le
5|\L_1|$. Lemma \ref{l:L_pm} implies that $k'$ in the sum
(\ref{2tmp_1}) belongs to a set of cardinality at most $8|\L_1|$.
Since $i\notin \Omega_1$, it follows that the set $E_1^{(2)}$ is
$\a_0$--uniform. Using Corollary \ref{c:sv}, we get
\[
    \sigma^{*}_{ij} \le \beta_1^{(2)} \sum_{k',m}  g(k'+m,m) ( \l_i * \o{E}_2^{(2)} ) (k') + 16 \a^{2/3}_0 |\L'|^2 |\L_1| \,.
\]
We have $j\notin \Omega^*_i$. Hence
\[
    \sigma^{*}_{ij} \le \beta_1^{(2)} \sum_{k',m}  g(k'+m,m) ( \l_i * E_2 ) (k') + 16 \a^{2/3}_0 |\L'|^2 |\L_1|
\]
\begin{equation}\label{2tmp_2}
        \le
            \beta_1^{(2)} \beta_2 |\L'| \sum_{k,m} g(k,m) + 32 \a^{1/6}_0 |\L'|^2 |\L_1| \,.
\end{equation}
Suppose that $\sigma^{*}_{ij} > 16 \d \beta^2_1 \beta^2_2 |\L'|^2
|\L_1|$. Since  $i \notin \Omega_1$, it follows that $\beta_1 /2
\le \beta_1^{(2)}  \le 2 \beta_1 $. Using this and (\ref{2tmp_2}),
we get
\begin{equation}\label{2tmp_3}
    \sum_{k,m} g(k,m) \ge 8 \beta^2_1 \beta^2_2 |\L'| |\L_1| \,.
\end{equation}
Recall that $m$ belongs to the set $\t{\L}_i +j$ in
(\ref{2tmp_3}). By Lemma \ref{l:L_pm}, we find
\begin{equation}\label{2tmp_4}
    \sum_{k,m} A(k,m) \L'(k-i) \L_1(m-i-j) \ge 4 \beta^2_1 \beta^2_2 |\L'| |\L_1| \,.
\end{equation}
We have $i\notin \Omega_1$ and $j \notin \Omega^*_i$. Using this
fact, inequality (\ref{2tmp_4}) and simple average arguments it is
easy to see that there is a vector $\v{y} = (y_1, y_2) \in G\m G$
and two sets $F_1 \subseteq E_1 \cap (\L' + y_1) $, $F_2 \subseteq
E_2 \cap (\L' + y_2)$ such that (\ref{est::delta'}),
(\ref{est::card'}) hold. This completes the proof of the lemma.
\end{proof}

We have
\begin{equation}\label{}
  |\sigma|^4 \le \| h \|_2^4 \cdot \sigma^{*} \cdot \sum_{m,u} \sum_{r,r'} f(r,m) f(r,u) f(r',m) f(r',u) \,\, \cdot
\end{equation}
\begin{equation}\label{}
                                      \sum_k E_1^{(2)} (k) E_1^{(2)} (k-m+u)
                                             E_2^{(1)} (m-k+r) E_2^{(1)} (m-k+r') =
\end{equation}
\begin{equation}\label{}
           =   \| h \|_2^4 \cdot \sigma^{*} \cdot \sum_{m,u} \sum_{r,r'} f(r,m) f(r,u) f(r',m) f(r',u) \,\, \cdot
\end{equation}
\begin{equation}\label{e:final_E}
                                      \sum_k E_1^{(2)} (m-k) E_1^{(2)} (u-k)
                                             E_2^{(1)} (k+r) E_2^{(1)} (k+r')
           = \| h \|_2^4 \cdot \sigma^{*} \cdot \sigma' \,.
\end{equation}
Rewrite  $\sigma'$ as
\begin{equation}\label{e:final_E1}
  \sigma' = \sum_k \sum_{r,r'} E_2^{(1)} (k+r) E_2^{(1)} (k+r')
                                                  | \sum_m E_1^{(2)} (m-k) f(r,m) f(r',m) |^2
\end{equation}
We have  $r\in \L_1$ and $k+r \in \L_2^{(1)}$. It follows that
$k\in \L_2^{(1)} - \L_1$. On the other hand $m-k \in \L_1^{(2)}$.
Hence $m\in \L_1^{(2)} + k \in \L_2^{(1)} + \L_1^{(2)} - \L_1$. By
symmetry  $u$ belongs to $\L_2^{(1)} + \L_1^{(2)} - \L_1$. Using
Lemma \ref{l:L_pm} for $\L_1$ and its $\eps_0$--attendant $\L'$,
we obtain that $k$ and $m,u$ belongs to some translations of {\it
Bohr } sets $W_1 = \L_1^{+}$ and $W_2 = W_1^{+}$, respectively,
and the cardinalities of these sets do not exceed $3|\L_1|$.

If $k$ is fixed, then $m,u,r,r'$ in (\ref{e:final_E}) run some
sets of the cardinalities at most $|\L'|$.

Let $\Phi^1_{r,r'} (m) = f(r,-m) f(r',-m) W_2(m-i-j)$,\\
$\Phi^2_{r,r'} (u) = f(r,-u) f(r',-u) W_2 (u-i-j)$,
$\Phi^3_{m,u} (r) = f(-r,m) f(-r,u)$, and \\
$\Phi^4_{m,u} (r') = f(r',m) f(r',u)$. Consider the sets
\[
 B_1 = \{ k ~|~ | (\Phi^1_{r,r'} * E_1^{(2)}) (-k) - \beta_1^{(2)} (\Phi^1_{r,r'} * \L_1^{(2)}) (-k) | \ge \a_0^{2/3} |\L'| \} \,
\]
\[
 B_2 = \{ k ~|~ | (\Phi^2_{r,r'} * E_1^{(2)}) (-k) - \beta_1^{(2)} (\Phi^2_{r,r'} * \L_1^{(2)}) (-k) | \ge \a_0^{2/3} |\L'| \} \,
\]
\[
 B_3 = \{ k \in \L_1 ~|~ | (\Phi^3_{m,u} * E_2^{(1)}) (k) - \beta_2^{(1)} (\Phi^3_{m,u} * \L_2^{(1)}) (k) | \ge \a_0^{2/3} |\L'| \} \,
\]
\[
 B_4 = \{ k \in \L_1 ~|~ | (\Phi^4_{m,u} * E_2^{(1)}) (k) - \beta_2^{(1)} (\Phi^4_{m,u} * \L_2^{(1)}) (k) | \ge \a_0^{2/3} |\L'| \} \,.
\]
We have $i \notin \Omega_1$, $j \notin \Omega_2$. Using Corollary
\ref{c:sv}, we get
$|B_1|, |B_2| \le 3 \a_0^{2/3} |\L_1|$ and $|B_3|, |B_4| \le
\a_0^{2/3} |\L_1|$. Let $B = B_1 \bigcup B_2 \bigcup B_3 \bigcup
B_4$. Then $|B| \le 8 \a_0^{2/3} |\L_1|$. Split $\sigma'$ as
\[
  \sigma' = \sum_{k\in B} \sum_{r,r'} E_2^{(1)} (k+r) E_2^{(1)} (k+r') | \sum_m E_1^{(2)} (m-k) f(r,m) f(r',m) |^2 +
\]
\[
           + \sum_{k \notin B} \sum_{r,r'} E_2^{(1)} (k+r) E_2^{(1)} (k+r') | \sum_m E_1^{(2)} (m-k) f(r,m) f(r',m) |^2
           = \sigma_1 + \sigma_2
\]
Let us estimate $\sigma_1$. Since  $|B| \le 8 \a_0^{2/3} |\L_1|$,
it follows that
\begin{equation}\label{}
  |\sigma_1| \le 8 \a_0^{2/3} |\L'|^4 |\L_1| \,.
\end{equation}
If $k\notin B$, then $k\notin B_1$. This implies that
\[
  \sigma_2 = \sum_{k\notin B} \sum_{u} \sum_{r,r'} f(r,u) f(r',u) E_1^{(2)} (u-k) E_2^{(1)} (k+r) E_2^{(1)} (k+r') \,\, \cdot
\]
\[
             \sum_m f(r,m) f(r',m) E_1^{(2)} (m-k) =
\]
\[
           = \sum_{k\notin B} \sum_{u} \sum_{r,r'} f(r,u) f(r',u) E_1^{(2)} (u-k) E_2^{(1)} (k+r) E_2^{(1)} (k+r')
            (\Phi^1_{r,r'} * E_1^{(2)}) (-k)
\]
\[
           = \beta_1^{(2)} \sum_{k\notin B} \sum_{u} \sum_{r,r'} f(r,u) f(r',u) E_1^{(2)} (u-k) E_2^{(1)} (k+r) E_2^{(1)} (k+r')
             \,\, \cdot
\]
\[
                                                \sum_m f(r,m) f(r',m) \L_1^{(2)} (m-k) +
\]
\[
           + \vartheta \a_0^{2/3} |\L'| \sum_{k\notin B} \sum_{u} \sum_{r,r'} f(r,u) f(r',u) E_1^{(2)} (u-k) E_2^{(1)} (k+r) E_2^{(1)} (k+r')
\]
\[
           = \beta_1^{(2)} \sum_{k\notin B} \sum_{u} \sum_{r,r'} f(r,u) f(r',u) E_1^{(2)} (u-k) E_2^{(1)} (k+r) E_2^{(1)} (k+r') \,\, \cdot
\]
\begin{equation}\label{}
                                                \sum_m f(r,m) f(r',m) \L_1^{(2)} (m-k)
           + 4 \vartheta \a_0^{2/3} |\L'|^4 |\L_1| \,,
\end{equation}
where $|\vartheta| \le 1$. Using these arguments for $B_2$, $B_3$
and $B_4$, we get
\[
  |\sigma_2| \le (\beta_1^{(2)})^2 (\beta_2^{(1)})^2
                                      \sum_{m,u} \sum_{r,r'} f(r,m) f(r,u) f(r',m) f(r',u) \,\, \cdot
\]
\begin{equation}\label{}
                                      \sum_k \L_1^{(2)} (m-k) \L_1^{(2)} (u-k)
                                             \L_2^{(1)} (k+r) \L_2^{(1)} (k+r')
                                      +
                                      16 \a_0^{2/3} |\L'|^4 |\L_1| \,,
\end{equation}
It follows that
\[
  |\sigma'| \le |\sigma_1| + |\sigma_2|
                                      \le (\beta_1^{(2)})^2 (\beta_2^{(1)})^2
                                      \sum_{m,u} \sum_{r,r'} f(r,m) f(r,u) f(r',m) f(r',u) \,\, \cdot
\]
\begin{equation}\label{}
                                      \sum_k \L_1^{(2)} (m-k) \L_1^{(2)} (u-k)
                                             \L_2^{(1)} (k+r) \L_2^{(1)} (k+r')
                                      +
                                      32 \a_0^{2/3} |\L'|^4 |\L_1| \,.
\end{equation}
Using (\ref{e:final_E}), we obtain
\[
  |\sigma|^4 \le \| h \|_2^4 \cdot \sigma^{*} \cdot (\beta_1^{(2)})^2 (\beta_2^{(1)})^2
                                                  \sum_k \sum_{r,r'} \L_2^{(1)} (k+r) \L_2^{(1)} (k+r') \,\, \cdot
\]
\begin{equation}\label{e:final_E_f}
                                                  \Big| \sum_m \L_1^{(2)} (m-k) f(r,m) f(r',m) \Big|^2
                                                  +
                                                  32 \| h \|_2^4 \cdot \sigma^{*} \cdot \a_0^{2/3} |\L'|^4 |\L_1|
\end{equation}
%
%
%
%
%
%
%
Since  $i \notin \Omega_1$, $j \notin \Omega_2$, it follows that
$\beta_1^{(2)} \le 2\beta_1$ and $\beta_2^{(1)} \le 2 \beta_2$.
Whence
\[
  |\sigma_{ij}|^4 \le 2^{4} \beta_1^2 \beta_2^2 \cdot \| h \|_2^4 \cdot \sigma_{ij}^{*} \cdot
  \sum_k \sum_{r,r'} \L_2^{(1)} (k+r) \L_2^{(1)} (k+r') \, \cdot \,
\]
\begin{equation}\label{}
  \Big| \sum_m \L_1^{(2)} (m-k) f(r,m) f(r',m) \Big|^2 + 2^{5} \a_0^{2/3} \cdot \| h \|_2^4 \cdot \sigma_{ij}^{*} \cdot |\L'|^4 |\L_1| \,.
\end{equation}
Let $\a_{ij} = \sum_k \sum_{r,r'} \L_2^{(1)} (k+r) \L_2^{(1)}
(k+r') | \sum_m \L_1^{(2)} (m-k) f(r,m) f(r',m) |^2$.

Suppose that there are $i\notin \Omega_1, j \notin \Omega'_2 (i)$
such that $\| h_j \|_2^2 \ge 8 \d \beta_1 \beta_2 |\L'| |\L_1|$.
It follows that
\begin{equation}\label{tmp:f_on_h}
    \sum_{k,m} A(k,m) \L'(m-j)  \ge 8 \d \beta_1 \beta_2 |\L'| |\L_1| \,.
\end{equation}
Let $F'_1 = E_1$, $F'_2 = E^{(1)}_2$. We have $j \notin \Omega'_2
(i)$. Using this and (\ref{tmp:f_on_h}), we get
\[
    |A \cap F'_1 \m F'_2| \ge 4 \d \beta_1 \beta_2 |F'_1| |F'_2|
\]
and
\[
    |F'_1| = \beta_1 |\L_1| \,, \quad |F'_2| \ge 2^{-1} \beta_2 |\L'| \,.
\]
Using simple average arguments it is easy to see that there are a
vector $\v{y} = (y_1, y_2) \in G\m G$ and two sets $F_1 \subseteq
E_1 \cap (\L' + y_1) $, $F_2 \subseteq E_2 \cap (\L' + y_2)$ such
that (\ref{est::delta'}), (\ref{est::card'}) hold.

Using Lemma \ref{l:w_est}, we obtain
\[
  \sum_{i\notin \Omega_1, j \notin \Omega'_2 (i)} |\sigma_{ij}|
        \le
            8 \d^{3/4} \beta^{3/2}_1 \beta^{3/2}_2 |\L'| |\L_1|^{3/4} \cdot
            \left( \sum_{i\in \L_1, j\in \L_2} \a_{ij} \right)^{1/4} (|\L_1| |\L_2|)^{3/4} \,+
\]
\[
                        +
                            4 \a_0^{1/6} |\L'|^2 |\L_1|^2 |\L_2| \,.
\]
By assumption the function $f$ is rectilinearly
$(\a,\eps)$--uniform. Clearly,
\[
    \sum_{i\in \L_1, j\in \L_2} \a_{ij}
        =
            \sum_{i\in \L_1, j \in \L_2}
           \sum_k \sum_{r,r'} \mu_j (k+r) \mu_j (k+r') \, \cdot \,
           \Big| \sum_m \l_i (m-k) f(r,m) f(r',m) \Big|^2 \,.
\]
It follows that
\[
  \sum_{i\notin \Omega_1, j \notin \Omega'_2 (i)} |\sigma_{ij}|
    \le
        8 \a^{1/4} \d^{3/4} \beta^{2}_1 \beta^{2}_2 |\L'|^{2} |\L_1|^2 |\L_2|
                +
\]
\begin{equation}\label{f:almost_final}
                +
                    4 \a_0^{1/6} |\L'|^2 |\L_1|^2 |\L_2| \,.
\end{equation}
Using  (\ref{f:4+R}), (\ref{e:sigma'}), (\ref{est:*}) and
(\ref{f:almost_final}), we have
\[
  |\sigma_0| \le 16 \k |\L_1|^2 |\L_2| + 8 \a_0^{1/2} |\L_1|^2 |\L_2| +
                 32 \a_0^{1/2} |\L_1|^2 |\L_2| + 4 \a_0^{1/6} |\L'|^2 |\L_1|^2 |\L_2|
\]
\[
                 + 8 \a^{1/4} \d^{3/4} \beta_1^2 \beta_2^2 |\L_1|^2 |\L_2|
             \le 16 \a^{1/4} \d^{3/4} \beta_1^2 \beta_2^2 |\L_1|^2 |\L_2|
\]
as required.
\end{proof*}

  The next result is the main in this section.

Let $\L_1$, $\L_2$ be Bohr sets, $\L_1 \le \L_2$, $\L_1 = \L
(S,\eps_1)$, $S \subseteq \F{G}$ and let $E_1 \subseteq \L_1$,
$E_2 \subseteq \L_2$, $|E_1| = \beta_1 |\L_1|$, $|E_2| = \beta_2
|\L_2|$. By $\mathcal{P}$ denote the product set $E_1 \m E_2$.

\begin{theorem}
Let  $A$ be an arbitrary subset of $E_1\m E_2$ of cardinality
$\delta |E_1||E_2|$. Suppose that  the sets $E_1,E_2$ are
$(\a_0,2^{-10} \eps^2)$--uniform, $\a_0 = 2^{-2000} \d^{96}
\beta_1^{48} \beta_2^{48}$, $\eps = (2^{-100}  \a_0^2 ) / (100
d)$. Let  $A$ be rectilinearly $(\a,\a_1,\eps)$--uniform, $\a =
2^{-100} \d^{9}$, $\a_1 = 2^{-7}$, and
\begin{equation}\label{COND:N}
   \log N \ge 2^{10} d \log \frac{1}{\eps_1 \eps} \,.
\end{equation}
Then either $A$ contains a triple $\{ (k,m), (k+d,m), (k,m+d) \}$,
where $d\neq 0$ or there exists a Bohr set $\t{\L}$, two sets
$F_1$, $F_2$ and a vector $\v{y} = (y_1, y_2) \in G\m G$, $F_1
\subseteq E_1 \cap (\t{\L} + y_1) $, $F_2 \subseteq E_2 \cap
(\t{\L} + y_2) $ such that $\t{\L}$ is an $2^{-4}
\eps^2$--attendant
of $\L_1$ and
\begin{equation}\label{est::delta''}
  \quad |F_1| \ge 2^{-20} \beta_1 |\t{\L}|,  \quad
        |F_2| \ge 2^{-20} \beta_2 |\t{\L}| \,\, \mbox{ and }
\end{equation}
\begin{equation}\label{est::card''}
  \quad  \d_{F_1 \m F_2} (A) \ge \frac{3}{2} \d \,.
\end{equation}
\label{a_case}
\end{theorem}
\begin{proof*}
Let $\L'$ be an $\eps$--attendant set of $\L_1$ to be chosen
later, and $\l_i = \L' + i$, $i \in \L_1$.
Let $G_i = (\l_i \m \L_2) \cap A$, $f_i (\v{s}) = f(s_1 + i, s_2)
\L'(s_1, s_2) $, $i\in \L_1$. By $G_i$ denote the characteristic
functions of the sets $G_i$. Let
\[
  B_1 = \{ i \in \L_1 ~|~ E_1 \cap \l_i  \mbox { is not } (8 \a_0^{1/4}, \eps) \mbox{--uniform}  \} \,,
\]
\[
  B_2 = \{ i \in \L_1 ~|~  | \d_{\l_i} (E_1) - \beta_1 | \ge 4 \a_0^{1/2}  \} \,,
\]
\[
  B_3 = \{ i \in \L_1 ~|~ \| f_i \|^4_{\L' \m \L_2,\eps} >
         \a \beta_1^2 \beta_2^2 |\L'_\eps|^4 |\L'|^2 |\L_2| \},
  \mbox{ and } B = B_1 \cup B_2 \cup B_3 \,.
\]
By assumption $E_1$ is $(\a_0,\eps)$--uniform. By Lemma
\ref{l:intermediate}, we get $|B_1| \le 8 \a_0^{1/4} |\L_1|$. and
$|B_2| \le 8 \a_0^{1/4} |\L_1|$. Since  $A$ is rectilinearly
$(\a,\a_1, \eps)$--uniform,  it follows that $|B_3| \le \a_1
|\L_1|$. Hence $|B| \le 16 \a_0^{1/4} |\L_1| + \a_1 |\L_1| \le 2
\a_1 |\L_1|$.

Using Lemma \ref{l:L_pm}, we obtain
\begin{equation}\label{tmp:20:08:1}
  A(\v{s}) = \frac{1}{|\L'|} \cdot \sum_{i\in \L_1} G_i (\v{s}) + \epsilon (\v{s}) \,,
\end{equation}
where
$\| \epsilon \|_{1} \le 2\k |\L_1| |\L_2|$, $\k =  \a_0^2$.
Consider the sum
\begin{equation}\label{}
  \sigma = \frac{1}{|\L'|} \sum_{i\in \L_1} \sum_{x,y} G_i (x+y,y) \,.
\end{equation}
We have $|A| = \d \beta_1 \beta_2 |\L_1| |\L_2|$. Using
(\ref{tmp:20:08:1}), we get
\begin{equation}\label{tmp_14:36.0}
\sigma \ge \frac{7\d \beta_1 \beta_2 }{8} |\L_1| |\L_2| \,.
\end{equation}
Split $\sigma$ as
\begin{equation}\label{tmp_14:36.0'}
  \sigma = \frac{1}{|\L'|}  \sum_{i\in B} \sum_{x,y} G_i (x+y,y)
            +
           \frac{1}{|\L'|} \sum_{i\notin B} \sum_{x,y} G_i (x+y,y)
         = \sigma_1 + \sigma_2 \,.
\end{equation}

Let us estimate $\sigma_1$.
We have
\begin{equation}\label{}
    \sigma_1
        =
            \frac{1}{|\L'|}  \sum_{i\in B_3 \setminus (B_1 \cup B_2)} \sum_{x,y} G_i (x+y,y)
                +
                    \frac{1}{|\L'|}  \sum_{i\in B_1 \cup B_2} \sum_{x,y} G_i (x+y,y)
                        \le
\end{equation}
\begin{equation}\label{}
                        \le
                            \frac{1}{|\L'|}  \sum_{i\in B_3 \setminus (B_1 \cup B_2)} \sum_{x,y} G_i (x+y,y)
                                +
                                    16 \a^{1/4}_0 |\L_1| |\L_2| \,.
\end{equation}
Suppose that there exists $i\notin B_1 \cup B_2$ such that
$$
    \sum_{x,y} G_i (x+y,y) \ge 4\d \beta_1 \beta_2 |\L'| |\L_2| \,.
$$
In other words
$$
    \sum_{x,y} G_i (x,y) \ge 4\d \beta_1 \beta_2 |\L'| |\L_2| \,.
$$
Put
$y_1 = i$, $y_2=0$ and $F_1 = (\L' + i) \cap E_1$. Since  $i
\notin B_2$, it follows that $|F_1| \ge \beta_1 |\L'| /2$.
Using simple average arguments we see that there exists an element
$a$ such that $F_2 = (\L' + a) \cap E_2$ has the cardinality at
least $\beta_2 |\tilde{\L}_1| /2$ and for $\v{y} = (i,a)$ we have
\[
  |A \cap (F_1 \m F_2) | > 2\d |F_1| |F_2| \,.
\]
Thus we get (\ref{est::delta''}), (\ref{est::card''}) and the
theorem is proven in the case.

We have $\a_1 = 2^{-7}$. Using $|B_3| \le \a_1 |\L_1|$ and
$\a_0^{1/4} \le 2^{-4} \a_1 \beta_1 \beta_2$, we obtain
\begin{equation}\label{}
    \sigma_1
        \le
            4 \d \beta_1 \beta_2 |\L'| |B_3| |\L_2|
                +   16 \a^{1/4}_0 |\L_1| |\L_2|
                    \le
                        2^{-3} \d \beta_1 \beta_2 |\L'| |\L_1| |\L_2| \,.
\end{equation}
Using this and (\ref{tmp_14:36.0}), (\ref{tmp_14:36.0'}), we
obtain
\begin{equation}\label{tmp_14:40}
  \frac{1}{|\L'|} \sum_{i\notin B} \sum_{x,y} G_i (x+y,y) \ge \frac{3\d \beta_1 \beta_2 }{4} |\L_1| |\L_2| \,.
\end{equation}
The formula (\ref{tmp_14:40}) implies that there exists $i_0\notin
B$ such that
\begin{equation}\label{tmp_14:40+}
  \sum_{x,y} G_{i_0} (x+y,y)  \ge
                                  \frac{3}{4} \d \beta_1 \beta_2 |\L'| |\L_2| \,.
\end{equation}
Let $G'(\v{s}) = G_{i_0} (\v{s})$. We have
\begin{equation}\label{tmp_14:47}
  \sum_k \sum_m G'(k+m,m) \ge 2^{-3} \d \beta_1 \beta_2 |\L'| |\L_2| \,.
\end{equation}
We have $m\in \L_2$ and $k+m \in \l_i$. It follows that $k\in \l_i
- \L_2$. Using Lemma \ref{l:L_pm} we obtain that $k$ belongs to  a
set of cardinality at most $2|\L_2|$. By the Cauchy--Schwartz
inequality, we get
\begin{equation}\label{}
  2^{-6} \d^2 \beta_1^2 \beta_2^2 |\L'|^2 |\L_2|^2 \le \sum_k \Big( \sum_m G'(k+m,m) \Big)^2 \cdot 2|\L_2| \,.
\end{equation}
It follows that
$$
  \sum_k \Big( \sum_m G'(k+m,m) \Big)^2
    =
        \sum_k \sum_{m,p} G'(k+m,m) G'(k+p,p)
            \ge
$$
\begin{equation}\label{e:main_term}
                                  \ge 2^{-7} \d^2 \beta_1^2 \beta_2^2 |\L'|^2 |\L_2| \,.
\end{equation}
Consider the sum
\begin{equation}\label{e:num_corners}
  \sigma_0 = \sum_{s_1,s_2,r} G'(s_1,s_2) G'(s_1 + r, s_2 + r) A(s_1, s_2 + r) \,.
\end{equation}
We have
$$
  G'(s_1,s_2) G'(s_1 + r, s_2 + r) f(s_1, s_2 + r) =
$$
\begin{equation}\label{tmp:kinder_g}
    = G'(s_1,s_2) G'(s_1 + r, s_2 + r) f_{i_0} (s_1, s_2 + r) \,,
\end{equation}
where $f_{i_0}$ is the restriction of the function $f$ to $G'$. It
follows that
$$
 \sigma_0 = \d \sum_{s_1,s_2,r} G'(s_1,s_2) G'(s_1 + r, s_2 + r) \mathcal{P} (s_1, s_2 + r)
            +
$$
$$
            +
            \sum_{s_1,s_2,r} G'(s_1,s_2) G'(s_1 + r, s_2 + r) f(s_1, s_2 + r) =
$$
\begin{equation}\label{tmp_15:15}
          = \d \sum_{s_1,s_2,r} G'(s_1,s_2) G'(s_1 + r, s_2 + r)
            +
            \sum_{s_1,s_2,r} G'(s_1,s_2) G'(s_1 + r, s_2 + r) f_{i_0} (s_1, s_2 + r) \,.
\end{equation}
The inequality (\ref{e:main_term}) implies that the first term in
(\ref{tmp_15:15}) is greater than \\
$2^{-7} \d^3 \beta_1^2 \beta_2^2 |\L'|^2 |\L_2|$.
Since  $i_0 \notin B$, it follows that $\| f_{i_0} \|^4 \le \a
\beta_1^2 \beta_2^2 |\L'|^2 |\L_2|$ and $\d_{\l_{i_0}} (E_1) \le 2
\beta_1 $. By assumption $\a = 2^{-100} \d^{9}$. Using Theorem
\ref{t:tmp_tha} and (\ref{tmp_14:40+}), we obtain that either the
second term in (\ref{tmp_15:15}) does not exceed
$$
    2^{10} \a^{1/4} \d^{3/4} \beta_1^2 \beta_2^2 |\L'|^2 |\L_2|
  \le
    2^{-8} \d^3 \beta_1^2 \beta_2^2 |\L'|^2 |\L_2|
$$
or there is a vector $\v{y} = (y_1, y_2) \in G\m G$ and two sets
$F_1 \subseteq E_1 \cap (\t{\t{\L}} + y_1) $, $F_2 \subseteq E_2
\cap (\t{\t{\L}} + y_2)$ such that (\ref{est::delta''}),
(\ref{est::card''}) hold. If we have the second situation then we
are done and $\t{\t{\L}}$ is an $\eps$--attendant of $\L'$. In the
other case $\sigma_0 \ge 2^{-7} \d^3 \beta_1^2 \beta_2^2 |\L'|^2
|\L_2|$.

The sum (\ref{e:num_corners}) is the number of triples $\{ (k,m),
(k+d,m), (k,m+d) \}$, where $k \in \L_{i_0}$, $m\in \L_2$, $d\in
G$. The number of triples with $d=0$ does not exceed $|\L'|
|\L_2|$. By assumption $\log N \ge 2^{10} d \log \frac{1}{\eps_1
\eps} $. Using Lemma \ref{l:Bohr_est}, we get $|\L'| > 2^8 ( \d^3
\beta_1^2 \beta_2^2 )^{-1}$. Hence, $2^{-8} \d^3 \beta_1^2
\beta_2^2 |\L'|^2 |\L_2| > |\L'| |\L_2|$. It follows that $A$
contains a triple $\{ (k,m), (k+d,m), (k,m+d) \}$ with $d\neq 0$.
This completes the proof.
\end{proof*}


\section{Non--uniform case.}
\label{sec_non-uniform_case}

\begin{lemma}
Let $\L_1$, $\L_2$ be Bohr sets, $\L_1 \le \L_2$, and $\L'$ be an
$\eps$--attendant set of $\L_1$, $\eps = \k / (100 d)$. Let set
$A$ be a subset of  $C \subseteq \L_1 \m \L_2$ of cardinality $\d
|C|$. By $B$ define the set of $s \in \L_1$ such that $|A\cap
((\L' + s) \m \L_2) | < (\d - \eta) |C\cap ((\L' + s) \m \L_2) |$,
where $\eta > 0$. Then
\[
  \sum_{s \in (\L_1 \setminus B) } |A\cap ((\L' + s) \m \L_2)| \ge
  \d \sum_{s \in (\L_1 \setminus B)} |C\cap ((\L' + s) \m \L_2)|
        +
\]
\[
        +
  \eta \sum_{s \in B} |C\cap ((\L' + s) \m \L_2)| - 4 \k |\L'| |\L_1| |\L_2| \,.
\]
\label{l:easy_case}
\end{lemma}

\begin{proof*}
Using Lemma \ref{l:L_pm}, we get
\begin{equation}\label{tmp_13:12;}
  \d |C| = \sum_{\v{s}} A(\v{s}) \L_1 (k) \L_2(m)  =
            \frac{1}{|\L'|} \sum_{n\in \L_1} \sum_{\v{s}} A(\v{s}) ((\L' + n) \m \L_2) (\v{s}) + 2 \vartheta \k |\L_1| |\L_2| \,,
\end{equation}
where $|\vartheta| \le 1$.
Split the sum (\ref{tmp_13:12;}) into a sum over $n \in B$ and a sum over $n \in \L_1\setminus B$.                
We have
\[
  \d |C| < \frac{1}{|\L'|} (\d - \eta) \sum_{n \in B} |C\cap ((\L' + n) \m \L_2)|
                +
\]
\begin{equation}\label{tmp:11113}
                +
            \frac{1}{|\L'|} \sum_{n \in (\L_1 \setminus B)} |A \cap ((\L' + n) \m \L_2)|
            +
            2 \k |\L_1| |\L_2| \,.
\end{equation}
In the same way
\begin{equation}\label{tmp:11132}
  |C| = \frac{1}{|\L'|} \sum_{n \in B} |C\cap ((\L' + n) \m \L_2)|
                +
            \frac{1}{|\L'|} \sum_{n \in (\L_1 \setminus B)} |C\cap ((\L' + n) \m \L_2)| + 2\vartheta_1 \k |\L_1| |\L_2| \,,
\end{equation}
where $|\vartheta_1| \le 1$. Combining  (\ref{tmp:11113}) and
(\ref{tmp:11132}), we obtain the required result.
\end{proof*}

Let $X$ be a finite set, $\mu$ be a measure on $X$ and let $Z : X
\to \mathbf{R}$ be a function. By $\mathbf{E} Z$ denote the sum
$\frac{1}{|X|} \sum_{x\in X} Z(x)$. The following lemma is
well--known (see e.g. \cite{Lacey}).

\begin{lemma}
    Let $p$ be a real number.
    Suppose that $Z : X \to [-1,1]$ is a function such that
    $\mathbf{E} Z = 0$ and $\mathbf{E} |Z|^p = \sigma^p$.
    Then
    \begin{equation}\label{f:inequality_Paley}
        \mu \left\{ x\in X ~:~ Z > \frac{\sigma^p}{5} \right\} \ge \frac{\sigma^p}{5} \,.
    \end{equation}
\label{l:Paley}
\end{lemma}
\begin{proof}
    Suppose that (\ref{f:inequality_Paley}) does not hold.
    Since $\mathbf{E} Z = 0$ it follows that
    $$
        -\mathbf{E} Z \mathbf{1}_{\{Z<0\}}
            = \mathbf{E} Z \mathbf{1}_{\{Z>0\}}
                \le
                    \mu \{ x ~:~ Z > 5^{-1} \sigma^p \} + \mathbf{E} Z \mathbf{1}_{\{0<Z\le 5^{-1} \sigma^p \}}
                        \le
                            \frac{2}{5} \sigma^p \,,
    $$
    where $\mathbf{1}_{\{Z<0\}}$, $ \mathbf{1}_{\{Z>0\}}$ are the characteristics functions of the sets
    $\{x ~:~ Z(x) < 0 \}$, $\{ x ~:~ Z(x) > 0\}$ respectively.
    We have $|Z(x)| \le 1$ for all $x\in X$.
    Hence
    \begin{equation}\label{}
        \sigma^p = \mathbf{E} |Z|^p = \mathbf{E} |Z|^p \mathbf{1}_{\{Z<0\}} + \mathbf{E} |Z|^p \mathbf{1}_{\{Z>0\}}
            \le 2 \mathbf{E} Z \mathbf{1}_{\{Z>0\}} \le \frac{4}{5} \sigma^p
    \end{equation}
    with contradiction.
\end{proof}

We need in the proposition concerning the properties of not
rectilinearly $\a$--uniform sets. The similar proposition was
proven in \cite{Shkr_tri,Shkr_tri_LMS_DAN,Green_BCC,Lacey}.

\begin{proposition}
Let $A$ be a subset of  $E_1\times E_2$ of cardinality $|A| =
\delta |E_1| |E_2|$. Suppose that $\a >0$ is a real number, $\a
\le \d^4 /8$, and $A$ is not rectilinearly $\a$--uniform. Then
there are two sets $F_1 \subseteq E_1$ and $F_2 \subseteq E_2$
such that
\begin{equation}\label{conj1+}
  |A\bigcap (F_1 \m F_2)| > (\d + 2^{-15} \cdot \min\{ \a^2 \d^{-5}, \a \d^{-2} \} ) |F_1||F_2| \quad
\mbox{ and }
\end{equation}
\begin{equation}\label{conj2+}
    |F_1| \ge 2^{-15} \min\{ \a^2 \d^{-5}, \a \d^{-2} \} \cdot |E_1|\,, \quad
    |F_2| \ge 2^{-15} \min\{ \a^2 \d^{-5}, \a \d^{-2} \} \cdot |E_2| \,.
\end{equation}
\label{na_case_pr+}
\end{proposition}
\begin{proof*}
    Denote by $f$ the balanced function of $A$.
    Suppose that
    \begin{equation}\label{3.04.12:45}
        \sum_{x} |\sum_y f(x,y) |^2 \le \a \d^{-2} |E_1| |E_2|^2 /16
    \end{equation}
    and
    \begin{equation}\label{3.04.12:45'}
        \sum_{y} |\sum_x f(x,y) |^2 \le \a \d^{-2} |E_1|^2 |E_2| /16 \,.
    \end{equation}
    If (\ref{3.04.12:45}) or (\ref{3.04.12:45'}) is not true then
    we can use Lemma \ref{l:Paley} and find two sets $F_1$, $F_2$
    such that (\ref{conj1+}), (\ref{conj2+}) hold.
    Let us prove that
    \begin{equation}\label{3.04.12:15}
        \| A \|^4 \ge (\d^4 + \a/2) |E_1|^2 |E_2|^2 \,.
    \end{equation}
    By assumption $\| f \|^4 \ge \a |E_1|^2 |E_2|^2$.
    Using the obvious formulas  $A = f + \d (E_1 \m E_2)$ and $\sum_{x,y} f(x,y) = 0$, we get
    \begin{equation}\label{3.04.1}
        \| A \|^4 \ge
                        (\d^4 + \a) |E_1|^2 |E_2|^2
                            +
    \end{equation}
    \begin{equation}\label{3.04.2}
                            +
                                \d \sum_{x,x',y,y'} f(x,y) f(x',y) f(x,y') +
                                \d \sum_{x,x',y,y'} f(x,y) f(x',y) f(x',y') +
    \end{equation}
    \begin{equation}\label{3.04.3}
                            +
                                \d \sum_{x,x',y,y'} f(x,y) f(x,y') f(x',y') +
                                \d \sum_{x,x',y,y'} f(x',y) f(x,y') f(x',y')
                            +
    \end{equation}
    \begin{equation}\label{3.04.4}
                            +
                                \d^2 \sum_{x,x',y,y'} f(x,y) f(x',y') +
                                \d^2 \sum_{x,x',y,y'} f(x',y) f(x,y') +
    \end{equation}
    \begin{equation}\label{3.04.5}
                                +\d^2 \sum_{x,x',y,y'} f(x,y) f(x',y) E_2 (y') +
                                \d^2 \sum_{x,x',y,y'} f(x,y) f(x,y') E_1 (x') +
    \end{equation}
    \begin{equation}\label{3.04.6}
                                +\d^2 \sum_{x,x',y,y'} f(x',y) f(x',y') E_1 (x) +
                                \d^2 \sum_{x,x',y,y'} f(x,y') f(x',y') E_2 (y) \,.
    \end{equation}
    It is easy to see that two summands in (\ref{3.04.4}) equal zero.
    Using (\ref{3.04.12:45}) and (\ref{3.04.12:45'}), we see that the sum of
    four terms in (\ref{3.04.5}) --- (\ref{3.04.6}) does not exceed $\a |E_1|^2 |E_2|^2 / 4$.
    Let us prove that any term in (\ref{3.04.2}) --- (\ref{3.04.3})
    at most $\a / (16 \d)$.
    Without loss of generality it can be assumed that
    the first summand in (\ref{3.04.2}) is greater than $\a / (16 \d)$.
    We have
    $$
        \frac{\a}{16 \d} \le \left( \sum_{x,y} |f(x,y)|^3 \right)^{1/3}
                                \cdot
                                    \left( \sum_{x,y} |\sum_{x'} f(x',y)|^{3/2} \cdot |\sum_{y'} f(x,y')|^{3/2} \right)^{2/3}
    $$
    $$
        \le 2 \d^{1/3} \cdot \left( \sum_{y} |\sum_{x'} f(x',y)|^{3/2} \cdot \sum_x |\sum_{y'} f(x,y')|^{3/2} \right)^{2/3}
    $$
    Thus, we have for example
    $$
        \sum_{y} |\sum_{x'} f(x',y)|^{3/2} \ge \frac{\a^{3/4}}{16 \d} \ge \frac{\a^2}{16 \d^5} \,.
    $$
    Using Lemma \ref{l:Paley} and find two sets $F_1$, $F_2$
    such that (\ref{conj1+}), (\ref{conj2+}) hold.
    So any term in (\ref{3.04.2}) --- (\ref{3.04.3})
    does not exceed $\a / (16 \d)$ and we have proved (\ref{3.04.12:15}).

    Let $e(x,y) = \{ (\o{x}, \o{y}) \in A ~|~ (\o{x}, y) \in A \mbox{ and } (x, \o{y}) \in A \}$
    and
    $N_x = \{ y ~|~ (x,y) \in A \}$, $N_y = \{ x ~|~ (x,y) \in A \}$.
    Clearly,
    \begin{equation}\label{norm_via_e(x,y)}
        \| A \|^4 = \sum_{(x,y) \in A} e(x,y) \,.
    \end{equation}
    Let $\t{X} = \{ x \in E_1 ~:~ | \sum_{y} f(x,y) | \le \a |E_2| / (32 \d^3) \}$
    and $\t{Y} = \{ y \in E_2 ~:~ | \sum_{x} f(x,y) | \le \a |E_1| / (32 \d^3)\}$.
    Let also $X^c = E_1 \setminus \t{X}$ and $Y^c = E_2 \setminus \t{Y}$.
    Note that $|X^c| \le \zeta |E_1|$, $|Y^c| \le \zeta |E_2|$, where $\zeta = \a/ (128 \d^2)$.
    Indeed, if $|X^c| > \zeta |E_1|$ then $\sum_x |\sum_{y} f(x,y)| > \a^2 |E_1| |E_2| / (2^{12} \d^5)$.
    Using Lemma \ref{l:Paley} and find two sets $F_1$, $F_2$
    such that (\ref{conj1+}), (\ref{conj2+}) hold.

    Let us prove that
    \begin{equation}\label{vtoraya_stadiya}
        \sum_{x\in \t{X}, y \in \t{Y}} A(x,y)  e(x,y) \ge (\d^4 + \a/4) |E_1|^2 |E_2|^2 \,.
    \end{equation}
    We have
    $$
        \| A \|^4 =
            \sum_{x\in \t{X}, y \in \t{Y}, x',y'} A(x,y) A(x',y) A(x,y') A(x',y')
                +
    $$
    $$
                +
            \sum_{x\in \t{X}, y \in Y^c, x',y'} A(x,y) A(x',y) A(x,y') A(x',y')
                +
    $$
    $$
                +
            \sum_{x\in X^c, y \in \t{Y}, x',y'} A(x,y) A(x',y) A(x,y') A(x',y')
                +
    $$
    $$
                +
            \sum_{x\in X^c, y \in Y^c, x',y'} A(x,y) A(x',y) A(x,y') A(x',y')
            = \sigma_0 + \sigma_1 + \sigma_2 + \sigma_3 \,.
    $$
    Clearly, $\sigma_3 \le |X^c| |Y^c| \cdot \d |E_1| |E_2| \le \a |E_1|^2 |E_2|^2 /16$.
    Further,
    $$
        \sigma_1
            \le
                |Y^c| \sum_{x,x',y'} A(x,y') A(x',y')
                    =
                        |Y^c| \sum_{y'\in \t{Y}} |\sum_x A(x,y') |^2
                            +
                        |Y^c| \sum_{y'\in Y^c} |\sum_x A(x,y') |^2
    $$
    $$
                            \le
                                4 \d^2 |Y^c| |E_1|^2 |E_2|
                                    +
                                        |Y^c|^2 |E_1|^2
                                            \le \a |E_1|^2 |E_2|^2 /16 \,.
    $$
    In the same way $\sigma_2 \le \a |E_1|^2 |E_2|^2 /16$.
    Using (\ref{3.04.12:15}) and (\ref{norm_via_e(x,y)}), we get (\ref{vtoraya_stadiya}).

    By (\ref{vtoraya_stadiya}), we find
    $(x_0,y_0) \in A \cap (\t{X} \m \t{Y})$ such that
    \begin{equation}\label{3.04.14:33}
        e(x_0,y_0) \ge (\d^3 + \frac{\a}{4\d}) |E_1| |E_2| \,.
    \end{equation}
    Put $F_1 = N_{y_0}$, $F_2 = ТN_{x_0}$.
    By definition of $\t{X}$, $\t{Y}$, we get
    $\big| |N_x| - \d |E_2| \big| \le \a |E_2| / (32 \d^3)$
    and
    $\big| |N_y| - \d |E_1| \big| \le \a |E_1| / (32 \d^3)$.
    In particular $|F_1|, |F_2| \ge \d/2$ and (\ref{conj1+}) holds.
    Obviously, $e(x,y) = | (N_y \m N_x) \cap A |$.
    Using (\ref{3.04.14:33}) and $\a \le \d^4/8$, we obtain
    $$
        |A \bigcap (F_1 \m F_2)| \ge (\d + \frac{\a}{4\d^3}) \left( 1+\frac{\a}{32\d^4} \right)^{-2} |F_1| |F_2| \ge
    $$
    $$
        \ge (\d + \frac{\a}{4\d^3}) (1-\frac{\a}{16\d^4}) |F_1| |F_2| \ge ( \d + \frac{\a}{8 \d^3} ) |F_1| |F_2| \,.
    $$
    and we get (\ref{conj2+}).
    This concludes the proof.
\end{proof*}

Let $\L_1$, $\L_2$ be Bohr sets, $\L_1 \le \L_2$, $\L_1 =
\L(S,\eps_0)$, $|S| = d$, and $E_1 \subseteq \L_1$, $E_2 \subseteq
\L_2$, $|E_1| = \beta_1 |\L_1|$, $|E_2| = \beta_2 |E_2|$. Let
$\mathcal{P}$ be a product set $E_1 \m E_2$.

\begin{theorem}
Let $A$ be a subset of $\mathcal{P}$ of cardinality $|A|=\delta
|E_1||E_2|$. Suppose that $A$ has no triples $\{ (k,m), (k+d,m),
(k,m+d) \}$ with $d\neq 0$, $E_1,E_2$ are $(\a_0,2^{-10}
\eps^2)$--uniform, $\a_0 = 2^{-2000} \d^{96} \beta_1^{48}
\beta_2^{48}$, $\eps = (2^{-100} \a_0^2) / (100 d)$, $\eps' =
2^{-10} \eps^2$, and
\[
  \log N \ge 2^{10} d \log \frac{1}{\eps_0 \eps} \,.
\]
Then there exists a Bohr set $\t{\L}$, two sets $F_1$, $F_2$ and a
vector $\v{y} = (y_1, y_2) \in G\m G$, $F_1 \subseteq E_1 \cap
(\t{\L} + y_1) $, $F_2 \subseteq E_2 \cap (\t{\L} + y_2) $ such
that
\begin{equation}\label{est::delta}
  \quad |F_1| \ge 2^{-500} \d^{22} \beta_1 |\t{\L}|,  \quad
        |F_2| \ge 2^{-500} \d^{22} \beta_2 |\t{\L}| \,\, \mbox{ and }
\end{equation}
\begin{equation}\label{est::card}
  \quad  \d_{F_1 \m F_2} (A) \ge \d + 2^{-500} \d^{22} \,.
\end{equation}
  Besides that for
  $\t{\L} = \L (\t{S},\t{\eps})$
  we have
  $\t{S} = S$ and $\t{\eps} \ge 2^{-5} \eps' \eps_0 $.
\label{t:Phase1}
\end{theorem}

\begin{proof*}
Let $\L'$ be an $\eps$--attendant of  $\L_1$, and $\L''$ be an
$\eps$--attendant of $\L'$ to be chosen later. Suppose that  $A$
is rectilinearly $(\a,\a_1,\eps)$--uniform, $\a = 2^{-100}
\d^{9}$, $\a_1 = 2^{-7}$. Using Theorem \ref{a_case}, we obtain
that either $A$ contains a triple $\{ (k,m), (k+d,m), (k,m+d) \}$
with $d\neq 0$ or (\ref{est::delta}), (\ref{est::card}) hold. At
the first case we get a contradiction, at the second case we
obtain the required result. Hence the set $A$ is not rectilinearly
$(\a,\a_1,\eps)$--uniform.

Let
\[
  B_1 = \{ s \in \L_1 ~|~ | \d_{\L'+s} (E_1) - \beta_1 | \ge 4\a_0^{1/2} \} \,,
\]
\[
  B_2 = \{ s \in \L_1 ~|~ \L' \cap (E_1 -s )\mbox {  is not } (8\a_0^{1/4},\eps) \mbox{--uniform}  \} \,,
\]
and
\[
  B = \{ i \in \L_1 ~|~ \| f_i \|^4_{\L' \m \L_2,\eps}
                            > \a \beta_1^2 \beta_2^2 |\L'_\eps|^4 |\L'|^2 |\L_2| \} \,.
\]
Since  $A$ is not rectilinearly $(\a,\a_1,\eps')$--uniform, it
follows that $|B| > \a_1 |\L_1|$. By assumption  $E_1$, $E_2$ are
$(\a_0,\eps')$--uniform. Using Lemma \ref{l:intermediate}, we
obtain $|B_1| \le 4\a_0^{1/2} |\L_1|$, $|B_2| \le 8\a_0^{1/2}
|\L_1|$. Let $B_3 = B_1 \cup B_2$. Then $|B_3| \le 12 \a_0^{1/2}
|\L_1|$. Let $B' = B \setminus B_3$. Since  $32 \a_0^{1/2} <
\a_1$, it follows that $|B'| \ge \a_1 |\L_1| /2$. Note that for
all $l\in B'$ we have
\begin{equation}\label{tmp_21:02}
  | \d_{\L'+s} (E_1) - \beta_1 | <  4\a_0^{1/2} \,.
\end{equation}

Let $\eta = 2^{-100} \a^{3/2}$. Let $\l_l = \L' + l$, $l\in \L_1$.
Suppose that for any $l \in B'$ we have
\begin{equation}\label{}
  | A \cap (\l_l \m \L_2) | \le (\d - \eta) |\l_l \cap E_1| |\L_2 \cap E_2| \,.
\end{equation}
Let $B'^{c} = \L_1 \setminus B'$. Using Lemma \ref{l:easy_case}
and  (\ref{tmp_21:02}), we get
\[
  \sum_{l \in B'^{c}} | A\cap (\l_l \m \L_2)| \ge
                                                \d |\L_2 \cap E_2| \sum_{l\in B'^{c}} |\l_l \cap E_1|
                                                +
                                                \eta |\L_2 \cap E_2| \sum_{l\in B'} |\l_l \cap E_1|
                                                - \a_0^2 |\L'| |\L_1| |\L_2|
\]
\[
                                                  \ge
                                                \d \beta_2 |\L_2| \sum_{l\in B'^{c}} |\l_l \cap E_1|
                                                +
                                                \eta \frac{\a_1 |\L_1|}{2} \frac{\beta_1 |\L'|}{4} \beta_2 |\L_2|
                                                   \ge
\]
\begin{equation}\label{tmp_21:14}
                                                   \ge
                                                \d \beta_2 |\L_2| \sum_{l\in B'^{c}} |\l_l \cap E_1|
                                                +
                                                2^{-3} \a_1 \eta \beta_1 \beta_2 |\L'| |\L_1| |\L_2|  \,.
\end{equation}
We have
\begin{equation}\label{tmp_21:39}
  \sum_{l \in B_1} | A\cap (\l_l \m \L_2)| \le 4 \a_0^{1/2} |\L_1| |\L'| |\L_2|
  \le 2^{-4} \a_1 \eta \beta_1 \beta_2 |\L'| |\L_1| |\L_2|  \,.
\end{equation}
Combining (\ref{tmp_21:14}) and (\ref{tmp_21:39}), we obtain
\begin{equation}\label{}
  \sum_{l \in (B'^{c} \setminus B_1)} | A\cap (\l_l \m \L_2)|
    \ge
        \d \beta_1 |\L_2| \sum_{l\in B'^{c}} |\l_l \cap E_1|
    +
        2^{-4} \a_1 \eta \beta_1 \beta_2 |\L'| |\L_1| |\L_2|  \,.
\end{equation}
This implies that, there exists a number  $l \in B'^{c} \setminus
B_1$ such that
\begin{equation}\label{tmp:eq_20:24}
  |A \cap (\l_l \m \L_2) | > (\d + 2^{-5} \a_1 \eta) |\l_l \cap E_1| |\L_2 \cap E_2| \,.
\end{equation}
Put $\tilde{\L} = \L'$, $y_1 = l_0$ and $F_1 = (\tilde{\L} + l_0)
\cap E_1$. Since  $l_0 \notin B_1$, it follows that $|F_1| \ge
\beta_1 |\tilde{\L}| /2$. The set $E_2$ is $(\a_0, 2^{-10}
\eps^2)$--uniform. This yields  that there exists a number $a$
such that $F_2 = (\tilde{\L} + a) \cap E_2$ has the cardinality at
least $\beta_2 |\tilde{\L}_1| /2$ and for $\v{y} = (l_0,a)$ we
have
\[
  |A \cap (\t{\L} + \v{y} ) | > (\d + 2^{-6} \a_1 \eta) |F_1| |F_2| \,.
\]
and the theorem is proven.

Let $f(\v{x})$ be the balanced function of $A$. There exists $l_0
\in B'$ such that
\[
  | A \cap (\l_{l_0} \m \L_2) | > (\d - \eta) |\l_{l_0} \cap E_1| |\L_2 \cap E_2| \,.
\]
If
\begin{equation}\label{tmp:20:25}
  | A \cap (\l_{l_0} \m \L_2) | \ge (\d + \eta) |\l_{l_0} \cap E_1| |\L_2 \cap E_2| \,,
\end{equation}
then the theorem is proven.

Hence there exists  $l_0 \in B'$ such that
\begin{equation}\label{e:<eta}
  | \sum_{r,m} f(r,m) \l_{l_0} (r) \L_2 (m) | < \eta |\l_{l_0} \cap E_1| |\L_2 \cap E_2| \,.
\end{equation}

Let $\L_0 = \L'+l_0$. Put $\nu_i = \L'' + i$, $i\in \L_0$ and
$\mu_j = \L'' + j$, $j\in \L_2$. Consider the sum
\begin{equation}\label{tmp_16:19}
  \sigma^* = \sum_{i\in \L_0} \sum_{j\in \L_2} \sum_k \sum_{m} \sum_{r\in \L_0} f(r,m) \nu_i (m-k) \mu_j (k+r) \,.
\end{equation}
Suppose that $i$ and $j$ are fixed in the sum (\ref{tmp_16:19}).
Using Lemma \ref{l:L_pm}, we obtain that $k$ runs a set of
cardinality at most $2|\L_0|$. Besides that if $i,j,k$ are fixed,
then $m$, $r$ run sets of size at most $|\L''|$. Using Lemma
\ref{l:L_pm} once again, we obtain
\begin{equation}\label{tmp_18:23}
   \sigma^* = |\L''|^2 \sum_k \sum_{m} \sum_{r\in \L_0} f(r,m) \L_0(m-k) \L_2(k+r) + \vartheta \a_0^2 |\L''|^2 |\L_0|^2 |\L_2| \,,
\end{equation}
where $|\vartheta| \le 1$. Let $\L_3 = \L_2 - \L' - l_0$. Using
Lemma \ref{l:L_pm}, we get $|\L_2| \le |\L_3| \le (1+\a_0^2)
|\L_2|$.
Note that $k$ belongs to  the set  $\L_3$  in (\ref{tmp_18:23}).                                                
If $k\in \L_2^{-} - l_0$, then $\L_2(k+r) = 1$, for all $r\in
\L_0$. If $k$ is fixed in (\ref{tmp_18:23}), then $r$ and $m$ run
sets of cardinality at most $|\L_0|$. It follows that
\[
  \frac{\sigma^*} {|\L''|^2}
            = \sum_{k\in (\L_2^{-} - l_0)} \sum_{m} \sum_{r\in \L_0} f(r,m) \L_0(m-k) +
\]
\[
            +
             \sum_{k\in (\L_3 \setminus (\L_2^{-} - l_0)) } \sum_{m} \sum_{r\in \L_0} f(r,m) \L_0(m-k) \L_2(k+r)
           =
\]
\[
           =
             \sum_{k\in (\L_2^{-} - l_0)} \sum_{m} \sum_{r\in \L_0} f(r,m) \L_0(m-k) + \a_0^2 \vartheta_1 |\L_0|^2 |\L_2|
           =
\]
$$
             \sum_k \sum_{m} \sum_{r\in \L_0} f(r,m) \L_0(m-k) + 2 \a_0^2 \vartheta_2 |\L_0|^2 |\L_2|
           =
$$
$$
           =
             |\L_0| \sum_{m} \sum_{r\in \L_0} f(r,m) + 2 \a_0^2 \vartheta_2 |\L_0|^2 |\L_2| \,,
$$
where $|\vartheta_1|, |\vartheta_2| \le 1$. Using (\ref{e:<eta}),
we get
\begin{equation}\label{tmp_22:47}
  |\sigma^*| < \eta |\L''|^2 |\L_0| |\L_0 \cap E_1| |\L_2 \cap E_2| + 4 \a_0^2 |\L''|^2 |\L_0|^2 |\L_2|
\end{equation}
If $j$ is fixed, then $k$  runs a set $-\L_0 + j + \L''$ in
(\ref{tmp_16:19}). Clearly, the cardinality of this set does not
exceed $(1 + \a_0^2) |\L'|$.
Hence, replacing $4 \a_0^2 |\L''|^2 |\L_0|^2 |\L_2|$ in (\ref{tmp_22:47})                                               
by $8 \a_0^2 |\L''|^2 |\L_0|^2 |\L_2|$, we can assume that $k$
runs $-\L_0 + j$ in (\ref{tmp_16:19}).

Since  $l\in B'$, it follows that $\beta_1 |\L_0| /2 \le |\L_0
\cap E_1| \le 2 \beta_1 |\L_0|$. Besides that $16 \a_0^2 < \eta
\beta_1 \beta_2$. This implies that
\[
  | \sum_{i\in \L_0} \sum_{j\in \L_2} \sum_{k\in -\L_0 + j} \sum_{m} \sum_{r\in \L_0} f(r,m) \nu_i (m-k) \mu_j (k+r) |
    <
\]
\begin{equation}\label{e:<eta2}
        <
            2 \eta |\L''|^2 |\L_0| \cdot |\L_0 \cap E_1| \cdot |\L_2 \cap E_2|
        \le 4 \eta \beta_1 \beta_2 |\L''|^2 |\L_0|^2 |\L_2| \,.
\end{equation}
\\

Let
\[
  \Omega = \{ j \in \L_2 ~|~
              \frac{1}{|\L'|} \sum_{k\in \L'+ j} | \d_{\L''+k} (E_2) - \beta_2 |^2 \ge 4 \a_0^{1/2}  \},
              \mbox{ and }
              G = \L_2 \setminus \Omega \,.
\]
Since  $E_2$ is $(\a_0,\eps')$--uniform, it follows that $|\Omega|
\le 8 \a_0^{1/2} |\L_2|$. Let $i\in \L_0$ be fixed. Let
\[
  \Omega (i)  = \{ j \in \L_2 ~|~
              \frac{1}{|\L'|} \sum_{k\in -\L_0+ j} | \d_{\L''+i+k} (E_2) - \beta_2 |^2 \ge 4 \a_0^{1/2}  \},
              \mbox{ and }
              G (i) = \L_2 \setminus \Omega(i) \,.
\]
Since
\[
    \sum_{k\in -\L_0+ j} | \d_{\L''+i+k} (E_2) - \beta_2 |^2 = \sum_{k\in \L'+ j + (i - l_0) } | \d_{\L''+k} (E_2) - \beta_2 |^2 \,,
\]
it follows that $\L_2 \cap ( G + l_0 -i) \subseteq G(i)$. Hence, $
 |\Omega(i)| \le |\L_2| - | \L_2 \cap ( G + l_0 - i) |
$. Since  $i$ belongs to  $\L_0$, this implies that a number
$a=l_0 - i$ belongs to  $\L'$. Using Lemma \ref{l:L_pm} for $\L_2$
and its $\eps$--attendant $\L'$, we get $(G \cap \L_2^{-}) + a
\subseteq \L_2$ and
\[
  |\L_2 \cap (G + a)| \ge |\L_2 \cap ( (G\cap \L_2^{-}) + a)| \ge | (G\cap \L_2^{-}) + a| =  | G\cap \L_2^{-} |
  \ge |G| - 8 \a_0^2 |\L_2| \,.
\]
Hence $|\Omega(i)| \le 8 \a_0^{1/2} |\L_2|$.

Since  $l_0 \in B'$, it follows that
 \begin{equation}\label{tmp:21:20'}
  \frac{1}{|\L'|} \sum_{k\in \L'} | \d_{\L''+k} (E_1 - l_0 \cap \L') - \beta_1 |^2 \le 2^6 \a_0^{1/2}
\end{equation}
It is clear that for any $j$ the sum (\ref{tmp:21:20'}) equals
\[
  \frac{1}{|\L'|} \sum_{k \in -\L_0 + j} |\d_{\L'' + j -k} (E_1 \cap \L_0) - \beta_1 |^2 \,.
\]
Indeed
\[
  \sum_{k \in -\L_0 + j} |\d_{\L'' + j -k} (E_1 \cap \L_0) - \beta_1 |^2
  =
  \sum_{k \in \L' + l_0} |\d_{\L'' + k} (E_1 \cap \L' + l_0) - \beta_1 |^2
  =
\]
\[
  =
  \sum_{k \in \L'} |\d_{\L'' + k} (E_1 - l_0 \cap \L') - \beta_1 |^2
\]

Let
\[
  \Omega_{1} (i,j) = \{ k \in -\L_0 + j ~:~ | \d_{\L''+i+k} (E_2) - \beta_2 | \ge 4 \a_0^{1/8} \} \,,
\]
\[
  \Omega_{2} (i,j) = \{ k \in -\L_0 + j ~:~ |\d_{\L'' + j -k} (E_1 \cap \L_0) - \beta_1 | \ge 4 \a_0^{1/8} \},
  \mbox{ and }
\]
\[
  \Omega_{3} (i,j) = \Omega_1 (i,j) \cup \Omega_2 (i,j) \,.
\]
For all $j\notin \Omega(i)$ we have $|\Omega_{1} (i,j)| \le 2
\a_0^{1/4} |\L'|$. The inequality (\ref{tmp:21:20'}) implies that
$|\Omega_2 (i,j)| \le 4 \a_0^{1/4} |\L'|$. Hence $|\Omega_3 (i,j)
| \le 8 \a_0^{1/4} |\L'|$ if $j\notin \Omega(i)$.
\\

Since  $l_0 \in B'$, it follows that
\[
  \sigma = \sum_{i\in \L_0} \sum_{j\in \L_2} \sum_k \sum_{m,u} \nu_i (m-k) \nu_i (u-k)
             \Big|\sum_r \mu_j (k+r) \t{f}_{l_0} (r,m) \t{f}_{l_0} (r,u) \Big|^2
             \ge
\]
\begin{equation}\label{e:nonunif}
             \ge
                \a \beta_1^2 \beta_2^2 |\L''|^4 |\L_0|^2 |\L_2| \,,
\end{equation}
where $\t{f}_{l_0}$ is a restriction of $f$ to $\l_{l_0} \m \L_2$.
If $j$ is fixed, then  $k$ runs $-\L_0 + j + \L''$  in
(\ref{e:nonunif}). Clearly, the cardinality of this set does not
exceed $(1 + \a_0^2) |\L'|$. Hence, replacing  $\a$ by $\a/2$ in
(\ref{e:nonunif}), we can assume that $k$ runs $-\L_0 + j$ in
(\ref{e:nonunif}). Using $|\Omega (i)| \le 8 \a_0^{1/2} |\L_2|$,
we get
\[
  \sigma = \sum_{i\in \L_0} \sum_{j\notin \Omega (i)} \sum_k \sum_{m,u} \nu_i (m-k) \nu_i (u-k)
             \Big|\sum_r \mu_j (k+r) \t{f}_{l_0} (r,m) \t{f}_{l_0} (r,u) \Big|^2
             \ge
\]
\begin{equation}\label{tmp_18:05}
             \ge
                \frac{\a}{4} \beta_1^2 \beta_2^2 |\L''|^4  |\L_0|^2 |\L_2| \,.
\end{equation}
\\                                                                                                

Now we can prove the theorem.

Let
\[
  J = \{
         (i,j,k) ~|~ i\in \L_0,\, j\notin \Omega(i),\, k \notin \Omega_3(i,j)
          \mbox{ such that  }
\]
\[
                                  \sum_{m,u} \nu_i (m-k) \nu_i (u-k)
                                  \Big|\sum_r \mu_j (k+r) \t{f}_{l_0} (r,m) \t{f}_{l_0} (r,u) \Big|^2
                                  \ge
                                  \frac{\a}{64} \beta_1^2 \beta_2^2 |\L''|^4
      \} \,.
\]
Using (\ref{tmp_18:05}), we get
\[
  \sum_{i \in \L_0} \sum_{j \notin \Omega(i)} \sum_{k\notin \Omega_3(i,j)} \sum_{m,u} \nu_i (m-k) \nu_i (u-k)
             \Big|\sum_r \mu_j (k+r) \t{f}_{l_0} (r,m) \t{f}_{l_0} (r,u) \Big|^2
             \ge
\]
\begin{equation}\label{tmp_15:05}
             \ge
                \frac{\a}{8} \beta_1^2 \beta_2^2 |\L''|^4  |\L_0|^2 |\L_2| \,.
\end{equation}

It follows that
\[
  \sum_{(i,j,k) \in J} \sum_{m,u} \nu_i (m-k) \nu_i (u-k)
             \Big|\sum_r \mu_j (k+r) \t{f}_{l_0} (r,m) \t{f}_{l_0} (r,u) \Big|^2
             \ge
\]
\begin{equation}\label{tmp_15:06}
             \ge
                \frac{\a}{16} \beta_1^2 \beta_2^2 |\L''|^4  |\L_0|^2 |\L_2| \,.
\end{equation}
Let us estimate the cardinality of $J$. For any triple $(i,j,k)$
belongs to $J$ we have $|E_2 \cap (\nu_i + k) | - \beta_2 |\L''| |
\le 4\a_0^{1/8} |\L''|$ and $|(E_1 \cap \L_0) \cap (\mu_j - k) | -
\beta_1 |\L''| | \le 4\a_0^{1/8} |\L''|$. Using (\ref{tmp_15:06}),
we get
\begin{equation}\label{}
  32 |J| \cdot |\L''|^4 \beta_1^2 \beta_2^2                                        
    \ge
  \frac{\a}{16} \beta_1^2 \beta_2^2 |\L''|^4  |\L_0|^2 |\L_2| \,.
\end{equation}
This yields that $|J| \ge 2^{-12} \a |\L_0|^2 |\L_2|$.

Let us assume that for all $(i,j,k) \in J$ we have
\begin{equation}\label{tmp_14:59.0}
  \sum_{m} \sum_{r\in \L_0} f(r,m) \nu_i (m-k) \mu_j (k+r) < - 2^{15} \frac{\eta}{\a} \beta_1 \beta_2 |\L''|^2 \,.
\end{equation}
Using (\ref{e:<eta2}), we get
\begin{equation}\label{}
  \sum_{(i,j,k) \in \o{J} } \sum_{m} \sum_{r\in \L_0} f(r,m) \nu_i (m-k) \mu_j (k+r) \ge
                                                    4 \eta \beta_1 \beta_2 |\L''|^2 |\L_0|^2 |\L_2| \,,
\end{equation}
where $\o{J} = \{ (i,j,k) ~:~ (i,j,k) \in (\L_0 \m \L_2 \m (-\L_0
+ j) ) \setminus J \}$. Since  $|\Omega(i)| \le 8 \a_0^{1/2}
|\L_2|$, $i\in \L_0$, it follows that
\begin{equation}\label{}
  \sum_{(i,j,k) \in \o{J}, j \notin \Omega(i)}
  \sum_{m} \sum_{r\in \L_0} f(r,m) \nu_i (m-k) \mu_j (k+r) \ge 2 \eta \beta_1 \beta_2 |\L''|^2 |\L_0|^2 |\L_2| \,.
\end{equation}
Hence, there exist $i$ and $j$, $j \notin \Omega (i)$ such that
\begin{equation}\label{}
  \sum_{k \in Q(i,j)} \sum_{m} \sum_{r\in \L_0} f(r,m) \nu_i (m-k) \mu_j (k+r) \ge \frac{\eta}{2} \beta_1 \beta_2  |\L'|^2 |\L_0| \,,
\end{equation}
where $Q(i,j)$ is a  subset of $-\L_0 + j$. Since  $j\notin
\Omega(i)$, it follows that $|\Omega_3(i,j)| \le 8 \a_0^{1/4}
|\L'|$. Hence
\begin{equation}\label{tmp_19:59}
  \sum_{k\in Q(i,j) \setminus \Omega_3(i,j) } \sum_{m} \sum_{r\in \L_0} f(r,m) \nu_i (m-k) \mu_j (k+r)
    \ge
  \frac{\eta}{4} \beta_1 \beta_2  |\L''|^2 |\L_0| \,.
\end{equation}
This implies that there exists $k\notin \Omega_3(i,j)$ such that
\begin{equation}\label{tmp_14:50}
  \sum_{m} \sum_{r\in \L_0} f(r,m) \nu_i (m-k) \mu_j (k+r)
    \ge
  \frac{\eta}{8} \beta_1 \beta_2  |\L''|^2 \,.
\end{equation}
Put $\tilde{\L} = \L''$,  $\v{y} = (j-k,k+i)$ and $F_1 =
(\tilde{\L} + y_1) \cap (E_1 \cap \L_0)$, $F_2 = (\tilde{\L} +
y_2) \cap E_2$. Since  $k\notin \Omega_3 (i,j)$, it follows that
$\beta_1 |\L''| /2 \le |F_1| \le 2 \beta_1 |\L''|$, $\beta_2
|\L''| /2 \le |F_2| \le 2 \beta_2 |\L''|$. Using this and
(\ref{tmp_14:50}), we get
\[
  |A \cap (F_1 \m F_2)| = |A \cap ( ( (\mu_j - k) \cap \L_0) \m ((\nu_i + k) \cap \L_2 )  ) |
        \ge
\]
\[
        \ge
            \d |(\mu_j - k) \cap E_1 \cap \L_0| |(\nu_i + k) \cap E_2| + \frac{\eta}{8} \beta_1 \beta_2  |\L''|^2
        \ge
\]
\[
        \ge
            ( \d
            +
            \frac{\eta}{32} ) |F_1| |F_2| \,.
\]
Hence, if for all $(i,j,k) \in J$ we have (\ref{tmp_14:59.0}),
then the theorem is proven.

Now assume that there exists a triple $(i,j,k ) \in J$ such that
\begin{equation}\label{}
  \sum_{m} \sum_{r\in \L_0} f(r,m) \nu_i (m-k) \mu_j (k+r) \ge -2^{15} \frac{\eta}{\a} \beta_1 \beta_2 |\L''|^2 \,.
\end{equation}
We can assume that for all  $(i,j,k) \in J$ we have
\begin{equation}\label{est:den_s}
  |\sum_{m} \sum_{r\in \L_0} f(r,m) \nu_i (m-k) \mu_j (k+r)| \le 2^{15} \frac{\eta}{\a} \beta_1 \beta_2 |\L''|^2 \,.
\end{equation}
Indeed, if
\[
  \sum_{m} \sum_{r\in \L_0} f(r,m) \nu_i (m-k) \mu_j (k+r) > 2^{15} \frac{\eta}{\a} \beta_1 \beta_2 |\L''|^2 \,,
\]
then we might apply the same reasoning as above. For sets
$\tilde{\L}_1 = \L''$, $\tilde{\L}_2 = \L''$, a vector $\v{y} =
(j-k,k+i)$ and $F_1 = (\tilde{\L}_1 + y_1) \cap (E_1 \cap \L_0)$,
$F_2 = (\tilde{\L}_2 + y_2) \cap E_2$ we have $|F_1| \ge \beta_1
|\tilde{\L}_1| /2$, $|F_2| \ge \beta_2 |\tilde{\L}_2| /2$
 and
\[
|A \cap (F_1 \m F_2) | \ge (\d + 2^{6} \frac{\eta}{\a} ) |F_1|
|F_2|.
\]
\\

Since  $(i,j,k) \in J$, it follows that
\begin{equation}\label{tmp_22:06}
  \sum_{m,u \in \nu_i +k }
                          \Big|\sum_{r \in \mu_j - k} \t{f}_{l_0} (r,m) \t{f}_{l_0} (r,u) \Big|^2
                               \ge
                          2^{-6} \a \beta_1^2 \beta_2^2 |\L''|^4 \,.
\end{equation}
Note that $m,u$ belong to $\nu_i+k \cap \L_2$ in (\ref{tmp_22:06})
and $r$ belongs to  a set $\mu_j - k \cap \L_0$. Put $
\mathcal{L}_1 = \mu_j - k \cap \L_0 $, $ \mathcal{L}_2 = \nu_i + k
\cap \L_2 $, $E_1' = E_1 \cap \mathcal{L}_1$ and  $E_2' = E_2 \cap
\mathcal{L}_2$.
We can assume that $\t{f}_{l_0}$ is zero outside $\mathcal{L}_1 \m
\mathcal{L}_2$ in (\ref{tmp_22:06}). Let $A_1 = A \cap
(\mathcal{L}_1 \m \mathcal{L}_2)$, $\delta_1 = \delta_{E_1' \m
E_2'} (A)$, and $f_1$ be a balanced function of $A_1$. Using
(\ref{est:den_s}), we get $|\d_1 - \d| \le 2^{20} \frac{\eta}{\a}
$. We have $k \notin \Omega_3 (i,j)$. Using this, we obtain
\begin{equation}\label{}
  \| \t{f}_{l_0} - f_1 \|^4 = |E_1'|^2 |E_2'|^2 ( \delta_1 - \delta )^2 \le
  2^{44} \beta_1^2 \beta_2^2 \frac{\eta^2}{\a^2} |\L''|^4 \,.
\end{equation}
Using Lemma \ref{l:norm}, we get
\begin{equation}\label{}
  \sum_{m,u \in \nu_i +k }
                          \Big|\sum_{r \in \mu_j - k} f_1 (r,m) f_1 (r,u) \Big|^2
                               \ge
                          2^{-7} \a \beta_1^2 \beta_2^2 |\L''|^4 \,.
\end{equation}
Since  $k \notin \Omega_3 (i,j)$, it follows that $2^{-1} \beta_1
|\L''| \le |E_1'| \le 2 \beta_1 |\L''|$, $2^{-1} \beta_2 |\L''|
\le |E_2'| \le 2 \beta_2 |\L''|$. Hence
\begin{equation}\label{}
  \sum_{m,u \in \nu_i +k }
                          \Big|\sum_{r \in \mu_j - k} f_1 (r,m) f_1 (r,u) \Big|^2
                               \ge
                          2^{-11} \a |E_1'|^2 |E_2'|^2 \,.
\end{equation}
Using Proposition \ref{na_case_pr+}, we obtain sets $F_1 \subseteq
E_1' \subseteq \mu_j - k$, $F_2 \subseteq E_2' \subseteq \nu_i +
k$ such that
\[
  |A \cap (F_1 \m F_2) | \ge
  |A_1 \cap (F_1 \m F_2) | \ge (\d_1 + 2^{-37} \frac{\a^2}{\d_1^5} ) |F_1| |F_2| \ge
\]
\begin{equation}\label{}
    \ge
    (\d + 2^{-40} \frac{\a^2}{\d^5}) |F_1| |F_2|
         \ge
            (\d + 2^{-240} \d^{13}) |F_1| |F_2| \,.
\end{equation}
and
\[
  |F_i| \ge 2^{-40} \frac{\a^2}{\d^5} |E_i'| \ge 2^{-300} \d^{13} \beta_i |\L''|, \quad i=1,2 \,.
\]
Put $\tilde{\L} = \L''$, $\v{y} = (j-k,k+i)$ and $F_1 =
(\tilde{\L}_1 + y_1) \cap (E_1 \cap \L_0)$, $F_2 = (\tilde{\L}_2 +
y_2) \cap E_2$. The sets $\tilde{\L}$ and $F_1$, $F_2$  satisfy
(\ref{est::delta}), (\ref{est::card}). This concludes the proof.
\end{proof*}


\section{On dense subsets of Bohr sets.}

The following lemmas were proven in \cite{Shkr_tri_LMS}.

\begin{lemma}
  Let $\L$ be a Bohr set, $\L'$ be an $\eps$--attendant of $\L$, $\eps = \k/(100 d)$,
  and $Q$ be a subset of $\L$.
  Let
  $g : 2^{G} \m (G\m G) \to {\bf D}$
  be the function
  such that
  $g (\L,\v{x}) = \d^2_{\L+\v{x}} (Q)$.
  Then
  \begin{equation}\label{}
   \frac{1}{|\L|^2} \sum_{\v{x} \in \L} g(\L', \v{x}) \ge g( \L,0) - 8\k \,.
  \end{equation}
\label{l:l4}
\end{lemma}

\begin{lemma}
  Let $\L$ be a Bohr set, $\L'$ be an $\eps$--attendant of $\L$, $\eps = \k/(100 d)$,
  $\a>0$ be a real number,
  and
  $Q$ be a subset of $\L$, $|Q| = \d |\L|$.
  Suppose that
  \begin{equation}\label{e:Big_sq}
    \frac{1}{|\L|^2} \sum_{\v{n} \in \L} | \d_{\L'+\v{n}} (Q) - \d |^2  \ge \a \,.
  \end{equation}
  Then
  \begin{equation}\label{}
    \sum_{\v{n} \in \L} \d^2_{\L'+\v{n}} (Q) \ge \d^2 + \a - 4 \k \,.
  \end{equation}
\label{l:l1}
\end{lemma}

\begin{note*}
Clearly, the one--dimension analogs of Lemma \ref{l:l4} and Lemma
\ref{l:l1} take place.
\end{note*}

Also, in \cite{Shkr_tri_LMS} was proven a corollary.

\begin{corollary}
  Let $\L$ be a  Bohr set, $\a>0$ be a real number,
  and $E_1$, $E_2$  be sets, $|E_1 \cap \L| = \beta_1 |\L|$, $|E_2 \cap \L| = \beta_2 |\L|$.
  Suppose that either $E_1$ or $E_2$ does not satisfy (\ref{c:sm_sq}).
  Let $\L'$ be an arbitrary  $(2^{-10} \a^2 \beta_1^2 \beta_2^2) /  (100 d)$--attendant set  of $\L$.
  Then
  \begin{equation}\label{}
    \frac{1}{|\L|^2} \sum_{\v{n} \in \L} \d_{\L'+\v{n}}^2 (E_1 \m E_2) \ge  \beta_1^2 \beta_2^2  ( 1 + \frac{\a^2}{2} ) \,.
  \end{equation}
\label{cor:1case_c}
\end{corollary}

The following lemma was proven by  J. Bourgain in \cite{Bu}. We
give his proof for the sake of completeness.

\begin{lemma}
  Let $\L = \L (S, \eps)$ be a Bohr set, $|S| = d \in {\bf N}$, $\a>0$ be a real number,
  and $Q$ be a set, $|Q \cap \L| = \d |\L|$.
  Suppose that
  \begin{equation}\label{e:Big_Fourier}
    \| (Q \cap \L - \d \L) \F{} ~ \|_{\infty}  \ge \a |\L| \,.
  \end{equation}
  Then there exists a Bohr set  $\L' = \L (S',\eps')$, $|S'| =d+1$ such that
  $\L'$ is  an $\eps_1$--attendant of $\L$,
  $\eps_1 = \frac{\k}{100 d}$, $\k\le \a/32$
  and
  \begin{equation}\label{}
    \frac{1}{|\L|} \sum_{n \in \L} | \d_{\L'+n} (Q) - \d |^2 \ge \frac{\a^2}{4} \,.
  \end{equation}
\label{l:l3}
\end{lemma}
\begin{proof*}
Let $Q_1 = Q \cap \L$. Using (\ref{e:Big_Fourier}), we obtain
\begin{equation}\label{}
  | \F{Q}_1 (\xi_0) - \d \F{\L} (\xi_0) | \ge \a |\L| \,,
\end{equation}
where $\xi_0 \in \F{G}$. We have $\L = \L_{S,\eps}$, where $S
\subseteq \F{G}$. Put $S' = S \cup \{ \xi_0 \} \subseteq \F{G}$
and
\[
  \L' = \L_{S', \eps'}
\]
be an $\eps_1$--attendant of $\L$. Using Lemma \ref{l:L_pm}, we
get
\[
  \F{Q}_1 (\xi_0) = \sum_n Q (n) \L(n) e^{-2\pi i (\xi_0 \cdot n)} = \frac{1}{|\L'|} \sum_n (\L * \L') (n) Q (n) e^{-2\pi i (\xi_0 \cdot n)}
                                                                + 2 \k \vartheta  |\L| \,,
\]
where $|\vartheta| \le 1$. We have
\[
  \F{Q}_1 (\xi_0) = \frac{1}{|\L'|} \sum_m \sum_n \L' (n-m) \L(m) Q(n) e^{-2\pi i (\xi_0 \cdot n)} +  2 \k \vartheta |\L|
                =
\]
\[
                =
                  \frac{1}{|\L'|} \sum_m \sum_n \L' (n-m) \L(m) Q(n) e^{-2\pi i (\xi_0 \cdot m)} +
\]
\[
                +
                  \frac{1}{|\L'|} \sum_m \sum_n \L' (n-m) \L(m) Q(n) [ e^{-2\pi i (\xi_0 \cdot n)} - e^{-2\pi i (\xi_0 \cdot m)} ] + 2 \k \vartheta  |\L|
                =
\]
\[
                =
                  \sum_{m \in \L} \d_{\L' + m} (Q) e^{-2\pi i (\xi_0 \cdot m)} +
                  \vartheta_1 \frac{1}{|\L'|}  \sum_m \sum_n \L' (n-m) \L(m) Q(n) | e^{-2\pi i (\xi_0 \cdot (n-m))} - 1|
                +
\]
\begin{equation}\label{tmp:19:09}
                +
                  2 \k \vartheta |\L|
                =
                  \sum_{m \in \L} \d_{\L' + m} (Q) e^{-2\pi i (\xi_0 \cdot m)} +
                  (14 \k \vartheta_1 + 2 \k \vartheta ) |\L|  \,,
\end{equation}
where $|\vartheta_1| \le 1$. Using (\ref{e:Big_Fourier}) and
(\ref{tmp:19:09}), we obtain
\begin{equation}\label{}
  \Big| \sum_{m \in \L} \d_{\L' + m} (Q) e^{-2\pi i (\xi_0 \cdot m)} - \d \sum_{m\in \L} e^{-2\pi i (\xi_0 \cdot m)} \Big| \ge \frac{\a}{2} |\L| \,.
\end{equation}
Hence
\begin{equation}\label{}
  \sum_{m \in \L} | \d_{\L' + m} (Q)  - \d | \ge \frac{\a}{2} |\L| \,.
\end{equation}
Using the Cauchy--Schwartz inequality, we get
\begin{equation}\label{}
    \frac{1}{|\L|} \sum_{\v{n} \in \L} | \d_{\L'+\v{n}} (Q) - \d |^2 \ge \frac{\a^2}{4} \,.
\end{equation}
This completes the proof.
\end{proof*}

\begin{corollary}
  Let $\L = \L (S,\eps)$ be a Bohr set, $\a>0$ be a real number,
  and $E_1$, $E_2$ be sets, $|E_1 \cap \L| = \beta_1 |\L|$, $|E_2 \cap \L| = \beta_2 |\L|$
  Suppose that either $E_1$ or $E_2$ satisfies (\ref{e:Big_Fourier}).
  Then there exists $(2^{-10} \a^2 \beta_1^2 \beta_2^2) /  (100 d)$--attendant set
  $\L' = \L (S',\eps')$ of the Bohr set $\L$ such that
  \begin{equation}\label{}
    \frac{1}{|\L|^2} \sum_{\v{n} \in \L} \d_{\L'+\v{n}}^2 (E_1 \m E_2) \ge  \beta_1^2 \beta_2^2  ( 1 + \frac{\a^2}{8} )
  \end{equation}
  and
  \begin{equation}\label{f:tmp_star}
    |S'| = d+1 \,.
  \end{equation}
\label{cor:together}
\end{corollary}
\begin{proof*}
Let $\v{n} = (x,y)$, and $\k = 2^{-10} \a^2 \beta_1^2 \beta_2^2$.
We have
\begin{equation}\label{tmp:22:45'}
  \frac{1}{|\L|^2} \sum_{\v{n} \in \L} \d_{\L'+\v{n}}^2 (E_1 \m E_2)
  =
  \Big( \frac{1}{|\L|} \sum_{x \in \L} \d_{\L' + x}^2 (E_1) \Big)
  \Big( \frac{1}{|\L|} \sum_{y \in \L} \d_{\L' + y}^2 (E_2) \Big)
\end{equation}
We can assume without loss of generality that $E_1$ satisfies
(\ref{e:Big_Fourier}). Using Lemma \ref{l:l3} and Lemma
\ref{l:l1}, we obtain
\begin{equation}\label{Tmp:22:47}
  \frac{1}{|\L|} \sum_{x \in \L} \d_{\L' + x}^2 (E_1) \ge \beta_1^2 + \frac{\a^2}{4} - 4 \k \,.
\end{equation}
Let us estimate the second term in (\ref{tmp:22:45'}). Using Lemma
\ref{l:l4}, we get
\begin{equation}\label{Tmp:22:47'}
  \frac{1}{|\L|} \sum_{y \in \L} \d_{\L' + y}^2 (E_2) \ge \beta_2^2 - 8 \k \,.
\end{equation}
Combining (\ref{Tmp:22:47}) and (\ref{Tmp:22:47'}), we obtain
\[
   \frac{1}{|\L|^2} \sum_{\v{n} \in \L} \d_{\L'+\v{n}}^2 (E_1 \m E_2)
        \ge
   ( \beta_1^2 + \frac{\a^2}{4} - 4 \k ) ( \beta_2^2 - 8 \k) \ge \beta_1^2 \beta_2^2  ( 1 + \frac{\a^2}{8} ) \,.
\]
This concludes  the proof.
\end{proof*}

We shall say that the set $S'$ from (\ref{f:tmp_star}) is
constructed by Corollary  \ref{cor:together}.

Clearly, all lemmas of this section apply to  translations of Bohr
sets.

Our further arguments and arguments from \cite{Shkr_tri_LMS} are
particulary the same.
\\

Let  $\mathbf{\L}$ be a union of a family of  Bohr sets 
$\L_0^*,  \L_1^* (\v{x}_0), \dots,  \L_{n}^* (\v{x}_0,\dots,
\v{x}_{n-1})$
and a sequence of some translations of Bohr sets $\Lambda_0,
\L_1(\v{x}_0), \dots, \L_{n} (\v{x}_0,\dots, \v{x}_{n-1})$ such
that
\[
  \L_1 (\v{x}_0) \mbox{ and }  \L_1^* (\v{x}_0)  \mbox{ are defined iff } \v{x}_0 \in \L_0              
\]
\[
  \L_2 (\v{x}_0, \v{x}_1) \mbox{ and }  \L_2^* (\v{x}_0, \v{x}_1) \mbox{ are defined iff } \v{x}_1 \in \L_1(\v{x}_0), \v{x}_0 \in \L_0
\]
\[
  \dots
\]
\[
  \L_n (\v{x}_0, \dots, \v{x}_{n-1} ) \mbox{ and }  \L_n^* (\v{x}_0, \dots, \v{x}_{n-1} ) \mbox{ are defined iff }
\]
\begin{equation}\label{c:L}
  \v{x}_{n-1} \in \L_{n-1} (\v{x}_0, \dots, \v{x}_{n-2}),
  \v{x}_{n-2} \in \L_{n-2} (\v{x}_0, \dots, \v{x}_{n-3}), \dots,
  \v{x}_0 \in \L_0 \,.
\end{equation}

Let   $m\ge 0$ be an integer number and $\mathbf{\L}$ be a family
of Bohr sets satisfies  (\ref{c:L}). Let  $g : 2^{G} \m ({G\m G})
\to {\bf D}$ be a function. Let us define the  {\it index} of $g$,
respect $\mathbf{\L}$, for all $k = 0,\dots, m$ by
\[
  \ind_k (\mathbf{\L}) (g)
    =
  \frac{1}{|\L_0|^2} \sum_{\v{x}_0 \in \L_0} \frac{1}{|\L_1(\v{x}_0)|^2} \sum_{\v{x}_1 \in \L_1 (\v{x}_0) }
    \dots
\]
\begin{equation}\label{i:k}
  \frac{1}{|\L_k(\v{x}_0, \dots, \v{x}_{k-1}) |^2}
  \sum_{\v{y} \in \L_k (\v{x}_0, \dots, \v{x}_{k-1}) }
  g(\L_k^* (\v{x}_0, \dots, \v{x}_{k-1}) ,\v{y}) \,.
\end{equation}

Let $M_k = M_k (\v{x}_0, \dots, \v{x}_{k-1})$  be the family of
sets such that $M_k (\v{x}_0, \dots, \v{x}_{k-1}) \subseteq \L_k
(\v{x}_0, \dots, \v{x}_{k-1})$ for all $(\v{x}_0, \dots,
\v{x}_{k-1})$. For any $k=0,\dots,m$ by  $\ind_k (\mathbf{\L},M)
(g)$ define the following expression
\[
  \ind_k (\mathbf{\L},M) (g) =
  \frac{1}{|\L_0|^2} \sum_{\v{x}_0 \in \L_0} \frac{1}{|\L_1(\v{x}_0)|^2} \sum_{\v{x}_1 \in \L_1 (\v{x}_0) }
    \dots
\]
\begin{equation}\label{i:k_M}
  \frac{1}{|\L_k(\v{x}_0, \dots, \v{x}_{k-1}) |^2}
  \sum_{\v{y} \in M_k (\v{x}_0, \dots, \v{x}_{k-1}) }
  g(\L_k^* (\v{x}_0, \dots, \v{x}_{k-1}) ,\v{y}) \,.
\end{equation}
Clearly, we have $| \ind_k (\mathbf{\L}, M) (g) | \le 1$, for any
natural $k\ge 0$, a family $M_k$ and a  function $g : 2^{G} \m
({G\m G}) \to {\bf D}$.

The following simple lemma was proven in \cite{Shkr_tri_LMS}.

\begin{lemma}
  Let $Q$ be a subset of $\L_0 \m \L_0$, and $|Q| = \delta |\L_0|^2$.
  Suppose that $\L_k^* (\v{x}_0, \dots, \v{x}_{k-1} )$ is an arbitrary $\eps$--attendant of
  $\L_k (\v{x}_0, \dots, \v{x}_{k-1} )$, $\eps = \k / (100 d)$.
  Let $ g(M,\v{x}) = \d_{M + \v{x}} (Q)$.
   Then for all $k=0,\dots, n$ we have
  \begin{equation}\label{}
    \Big| \ind_k (\mathbf{\L}) (g) - \d \Big| \le  4 \k (k+1)   \,.
  \end{equation}
\label{l:keps}
\end{lemma}

   The next result is the main in this section.

\begin{proposition}
Let $\L= \L(S, \eps_0)$ be a Bohr set, $|S| = d$, and $\v{s} =
(s_1,s_2)$ be a vector. Let $\eps, \sigma,\tau, \d \in (0,1)$ be
real numbers, $E_1$, $E_2$ be sets, $E_i = \beta_i |\L|$, $i=1,2$.
Suppose that  ${\bf E} = E_1 \m E_2$ is a subset of  $(\L+s_1) \m
(\L+s_2)$, $A \subset {\bf E}$, $\d_{{\bf E}} (A) = \d + \tau$,
and $\eps \le \k/(100 d)$, $ \k = 2^{-100} ( \tau \beta_1 \beta_2
)^5 \sigma^3 $. Let
\begin{equation}\label{c:Nge}
  N \ge (2^{-100} \eps_0 \eps)^{ - 2^{100} ( (\tau \beta_1 \beta_2)^{-5} \sigma^{-3} + d )^2  }  \,,
\end{equation}
and $
 ~\sigma \le 2^{-100} \tau \beta_1 \beta_2
$. Then there exists a Bohr set $\L'= \L(S',\eps')$, $|S'| = D$, $
  D\le 2^{30} (\tau \beta_1 \beta_2)^{-5} \sigma^{-3} + d
$, $
  \eps' \ge (2^{-10} \eps)^D \eps_0
$
and a vector $\v{t} = (t_1,t_2)$ such that if $E_1' = (E_1 - t_1)
\cap \L'$, $E_2' = (E_2 - t_2) \cap \L'$, ${\bf E'} = E_1' \m
E_2'$,
then                                                                                                    
\\
$ 1)~ |{\bf E'}| \ge \beta_1 \beta_2 \tau |\L'| /16$; \\
$ 2)~ E_1', E_2' $ are $(\sigma,\eps)$--uniform subsets of $\L'$; \\
$ 3)~ \d_{ {\bf E'} } (A - \v{t}) \ge \d + \tau/16 .$
\label{prop:point_3}
\end{proposition}

\begin{proof*}
Let $\beta = \beta_1 \beta_2$, and $\t{E}_1 = E_1 - s_1$, $\t{E}_2
= E_2 -s_2$, $\t{E} = \t{E}_1 \m \t{E}_2$.
If the sets $\t{E}_1$, $\t{E}_2$ are $(\sigma,\eps)$--uniform
subsets of $\L$, then Proposition \ref{prop:point_3} is proven.

Suppose that $\t{E}_1$, $\t{E}_2$ are not
$(\sigma,\eps)$--uniform subsets of  $\L$. We shall construct a
family of Bohr sets $\mathbf{\L}$ such that $\mathbf{\L}$
satisfies the conditions (\ref{c:L}). The proof of Proposition
\ref{prop:point_3} is a sort of an algorithm. At the first step of
our algorithm we put $\L_0 = \L = \L (S,\eps_0)$. If either
$\t{E}_1$ or $\t{E}_2$ does not satisfy (\ref{c:sm_total_F}) with
$\alpha = \sigma/2$, then
let $\L_0^*$ be an $\eps$--attendant of  $\L_0$ such that $\L_0^*$ is constructed by Corollary  \ref{cor:together}.        
In the other cases let $\L_0^*$ be an $\eps$--attendant of  $\L_0$
with the same set $S$ to be chosen later. Define
\[
  R_0 = \{ \v{p} = (p_1,p_2) \in \L_0 ~|~ \t{E}_1 - p_1, \t{E}_2 - p_2 \mbox { are } (\sigma,\eps)\mbox{--uniform in } \L_0^*
\]
\[
           \mbox{ or } \d_{\L_0^* + \v{p} } (\t{E}_1 \m \t{E}_2) < \beta \tau /16
        \}
\]
and $\o{R}_0 = (\L_0 \m \L_0) \setminus R_0$.

Let $\tilde{\L}$ be an arbitrary Bohr set, and $\v{n} \in G\m G$
be an arbitrary vector. Put  $g(\tilde{\L}, \v{n}) = \d^2_{
\tilde{\L} + \v{n}} ( \t{E})$, $g_1 (\t{\L}, \v{x}) =
\d_{\t{\L}+\v{n}} (A)$, $ g_2 (\t{\L},\v{n}) = \d_{\t{E} \cap
\t{\L} + \v{n} } (A) $ and $ g_3(\t{\L}, \v{n}) = \d_{\t{\L} +
\v{n}} (\t{E})$. Clearly,  $g(\tilde{\L}, \v{n}) = g^2_3 (\t{\L},
\v{n})$ and $ g_1 (\t{\L}, \v{x}) \le g_3 (\t{\L}, \v{n})$.
Besides that, we have
\[
  g_1 (\t{\L},\v{n}) = g_2 (\t{\L}, \v{n}) g_3(\t{\L}, \v{n}) \,.
\]

Let $\mathbf{\L}_0 = \{ \L_0 \}$. If $\ind_0
(\mathbf{\L}_0,\o{R}_0) (g_3) < \tau \beta /4$, then we stop the
algorithm at step $0$.

Using Lemma  \ref{l:smooth_d} and the Cauchy--Schwartz inequality,
we get
\begin{equation}\label{}
  \ind_0 (\mathbf{\L}_0 ) (g) \ge
         \Big( \frac{1}{|\L_0|^2} \sum_{\v{y} \in \L_0} \d_{\L_0^* + \v{y}} (\t{E}) \Big)^2
            \ge
            \beta / 2 \,.
\end{equation}

Let after the $k$th step of the algorithm the family of Bohr sets
$\mathbf{\L}_k$ has been constructed, $k\ge 0$.

Let
\[
  \L_{k+1} (\v{x}_0, \dots, \v{x}_{k} ) = \L_k^* (\v{x}_0, \dots, \v{x}_{k-1}) + \v{x}_k \,, \,
   \v{x}_k \in \L_k (\v{x}_0, \dots, \v{x}_{k-1}) \,.
\]
Let $\v{x}_k = (a,b)$, and $\L_k^* = \L_k^* (\v{x}_0, \dots,
\v{x}_{k-1})$. If either $(\t{E}_1 - a) \cap \L_k^* $ or $(\t{E}_2
- b) \cap \L_k^*$ does not satisfy  (\ref{c:sm_total_F}) with $\a
= \sigma/2$, then let $\L_{k+1}^* (\v{x}_0, \dots, \v{x}_{k} )$ be
an $\eps$--attendant of $\L_{k}^* (\v{x}_0, \dots, \v{x}_{k} )$
such that $\L_{k+1}^* (\v{x}_0, \dots, \v{x}_{k} )$ is constructed
by Corollary  \ref{cor:together}. In the other cases let
$\L_{k+1}^* (\v{x}_0, \dots, \v{x}_{k} )$ be an $\eps$--attendant
of $\L_{k}^* (\v{x}_0, \dots, \v{x}_{k} )$ with the same
generative vector.

By $R_{k+1} (\v{x}_0, \dots, \v{x}_{k})$, $\o{R}_{k+1} (\v{x}_0,
\dots, \v{x}_{k})$ denote the sets
\[
  R_{k+1} (\v{x}_0, \dots, \v{x}_{k}) =
            \{ \v{p} = (p_1,p_2) \in \L_k^* (\v{x}_0, \dots, \v{x}_{k-1}) + \v{x}_k
            ~|~ \t{E}_1 - p_1, \t{E}_2 - p_2
\]
\[
   \mbox { are }
   (\sigma,\eps)\mbox{--uniform in } \L_{k+1}^* (\v{x}_0, \dots, \v{x}_{k} )
\]
\[
           \mbox{ or } \d_{\L_{k+1}^* (\v{x}_0, \dots, \v{x}_{k} ) + \v{p} } (\t{E}_1 \m \t{E}_2) < \tau \beta /16
        \}
\]
and $
  \o{R}_{k+1} (\v{x}_0, \dots, \v{x}_{k}) =
        ( \L_k^* (\v{x}_0, \dots, \v{x}_{k-1}) + \v{x}_k ) \setminus R_{k+1} (\v{x}_0, \dots, \v{x}_{k})
$.

By $E_{k} (\v{x}_0, \dots, \v{x}_{k-1})$ denote the sets
\[
  E_{k} (\v{x}_0, \dots, \v{x}_{k-1}) =
\]
\[
  \{
    \v{p} = (p_1,p_2) \in \L_{k-1}^* (\v{x}_0, \dots, \v{x}_{k-2}) + \v{x}_{k-1}
            ~|~
                \d_{\L_{k}^* (\v{x}_0, \dots, \v{x}_{k-1} ) + \v{p} } (\t{E}_1 \m \t{E}_2) < \tau \beta /16
  \}.
\]
Obviously, $E_{k} (\v{x}_0, \dots, \v{x}_{k-1}) \subseteq R_{k}
(\v{x}_0, \dots, \v{x}_{k-1})$, $k=0,1,\dots$

Let $\mathbf{\L'}_{k+1} = \{ \L_{k+1} (\v{x}_0, \dots, \v{x}_{k})
\}$, $\v{x}_k \in \L_k (\v{x}_0, \dots, \v{x}_{k-1})$, and
$\mathbf{\L}_{k+1} = \{ \mathbf{\L}_k , \mathbf{\L'}_{k+1} \}$. If
$\ind_{k+1} (\mathbf{\L}_{k+1}, \o{R}_{k+1}) (g_3) < \tau \beta
/4$, then we stop the algorithm at step $k+1$.

Let $\L_{k-1}^* = \L_{k-1}^* (\v{x}_0, \dots, \v{x}_{k-2})$, and $
\beta_k' = \d_{\L_{k-1}^*} (\t{E}_1)$, $ \beta_k'' =
\d_{\L_{k-1}^*} (\t{E}_2)$. Suppose $\v{x}_{k-1} = (a',b')$
belongs to  $\o{R}_{k-1} (\v{x}_0, \dots, \v{x}_{k-2})$. Note that
$\v{x}_{k-1}$ does not belong to $E_{k-1} (\v{x}_0, \dots,
\v{x}_{k-2})$.
Let us consider three cases. \\
Case 1 :  either $( \t{E}_1 - a') \cap \L_{k-1}^*$ or
$(\t{E}_2 - b') \cap \L_{k-1}^*$ does not satisfy (\ref{c:sm_B}). \\
Case 2 :  either $( \t{E}_1 - a') \cap \L_{k-1}^*$ or
$(\t{E}_2 - b') \cap \L_{k-1}^*$ does not satisfy (\ref{c:sm_sq}). \\
Case 3 :  either $( \t{E}_1 - a') \cap \L_{k-1}^*$ or
$(\t{E}_2 - b') \cap \L_{k-1}^*$ does not satisfy (\ref{c:sm_total_F}). \\
Note that  $\a$  equals $\sigma$ in all these cases. \\
Let us consider the following situation  :  either $(\t{E}_1 - a')
\cap \L_{k-1}^*$ or $(\t{E}_2 - b') \cap \L_{k-1}^*$ does not
satisfy (\ref{c:sm_total_F}) with $\a= 2^{-4} \sigma^{3/2}$. Let
\begin{equation}\label{}
  S_0 =
       \frac{1}{ |\L_k (\v{x}_0, \dots, \v{x}_{k-1}) |^2 }
       \sum_{\v{y} \in \L_k (\v{x}_0, \dots, \v{x}_{k-1})} g(\L^*_k (\v{x}_0, \dots, \v{x}_{k-1}), \v{y}) \,,
\end{equation}
where $\L_{k}^* (\v{x}_0, \dots, \v{x}_{k-1} )$ is an
$\eps$--attendant of  $\L_{k} (\v{x}_0, \dots, \v{x}_{k-1} )$ such
that $\L_{k}^* (\v{x}_0, \dots, \v{x}_{k-1} )$ is constructed by
Corollary  \ref{cor:together}. Using Corollary \ref{cor:together},
we get
\[
   S_0
        \ge
            g ( \L_k (\v{x}_0, \dots, \v{x}_{k-1}), 0) ( 1 + 2^{-11} \sigma^3 )
        =
\]
\begin{equation}\label{}
        =
            g ( \L_{k-1}^* (\v{x}_0, \dots, \v{x}_{k-2}), \v{x}_{k-1} ) ( 1 + 2^{-11} \sigma^3 ) \,.
\end{equation}
Note that in this case, we have $\dim \L^*_k (\v{x}_0, \dots,
\v{x}_{k-1}) = \dim \L_k (\v{x}_0, \dots, \v{x}_{k-1}) + 1$.

Suppose that either $(\t{E}_1 - a') \cap \L_{k-1}^*$ or $(\t{E}_2
- b') \cap \L_{k-1}^*$ does not satisfy (\ref{c:sm_sq}) with $\a =
2^{-4} \sigma^{3/2}$. Using Corollary \ref{cor:1case_c}, we
obtain
\begin{equation}\label{}
  S_0
        \ge
              g ( \L_{k-1}^* (\v{x}_0, \dots, \v{x}_{k-2}), \v{x}_{k-1} ) ( 1 + 2^{-11} \sigma^3 ) \,.
\end{equation}
In this case, we have $\dim \L^*_k (\v{x}_0, \dots, \v{x}_{k-1}) =
\dim \L_k (\v{x}_0, \dots, \v{x}_{k-1})$.

Finally, suppose that either $(\t{E}_1 - a') \cap \L_{k-1}^*$ or
$(\t{E}_2 - b') \cap \L_{k-1}^*$ does not satisfy (\ref{c:sm_B})
with $\a = \sigma$. Note that $(\t{E}_1 - a') \cap \L_{k-1}^*$
{\it and} $(\t{E}_2 - b') \cap \L_{k-1}^*$ satisfy (\ref{c:sm_sq})
with $\a = 2^{-4} \sigma^{3/2}$. Let $\L_{k}^* = \L_{k}^*
(\v{x}_0, \dots, \v{x}_{k} )$. Define
\[
  B_k (\v{x}_0, \dots, \v{x}_{k-1}) = \{ \v{p} = (p_1,p_2) \in \L_k (\v{x}_0, \dots, \v{x}_{k-1})
            ~:~
\]
\[
                \| ( (\t{E}_1 - p_1)  - \beta_k' \L_{k}^* )\F{}~ \|_{\infty} \ge \sigma |\L_{k}^*|
                    \mbox{ or }
                    \| ( (\t{E}_2 - p_2)  - \beta_k'' \L_{k}^* )\F{}~ \|_{\infty} \ge \sigma |\L_{k}^*| \} \,.
\]
We have
\begin{equation}\label{tmp:12:51}
    | B_k (\v{x}_0, \dots, \v{x}_{k-1}) | \ge  \sigma |\L_k (\v{x}_0, \dots, \v{x}_{k-1})|^2 \,.
\end{equation}
Let
\[
  \t{B}_k (\v{x}_0, \dots, \v{x}_{k-1}) = \{ \v{p} = (p_1,p_2) \in B_k (\v{x}_0, \dots, \v{x}_{k-1})
                                                ~:~
\]
\[
                                              | \d_{\L_{k}^*} (\t{E}_1 - p_1) - \beta_k' | \le \sigma / 8
                                                \quad \mbox { and } \quad
                                              | \d_{\L_{k}^*} (\t{E}_2 - p_2) - \beta_k'' | \le \sigma / 8
                                          \} \,.
\]
For all $ \v{p} \in \t{B}_k (\v{x}_0, \dots, \v{x}_{k-1})$, we
have either $
  (\t{E}_1 - p_1) \cap \L_{k}^*
$ or $
 (\t{E}_2 - p_2) \cap \L_{k}^*
$ does not $\sigma/2$--uniform. The sets  $(\t{E}_1 - a') \cap
\L_{k-1}^*$ and  $(\t{E}_2 - b') \cap \L_{k-1}^*$ satisfy
(\ref{c:sm_sq}) with $\a$ equals $2^{-4} \sigma^{3/2}$. This
implies that
\begin{equation}\label{tmp:12:51'''}
  | \t{B}_k (\v{x}_0, \dots, \v{x}_{k-1})|
                \ge  \frac{\sigma}{2} |\L_k (\v{x}_0, \dots, \v{x}_{k-1})|^2 \,.
\end{equation}

Suppose that
\begin{equation}\label{c:g_3}
  g_3 (\L_{k-1}^*, \v{x}_{k-1}) = \beta_k' \beta_k'' \ge \tau \beta /8 \,.
\end{equation}
It follows from (\ref{c:g_3}) that
\begin{equation}\label{c:g_3'}
  g_3 (\L^*_k (\v{x}_0, \dots, \v{x}_{k-1}), \v{p}) \ge \beta_k' \beta_k'' - \sigma /2 \ge \tau \beta /16 \,,
\end{equation}
for all $\v{p} \in \t{B}_k (\v{x}_0, \dots, \v{x}_{k-1})$.

Let us consider the sum
\[
  S = S (\v{x}_0, \dots, \v{x}_{k-1})
        =
  \frac{1}{ |\L_k (\v{x}_0, \dots, \v{x}_{k-1}) |^2 }
    \sum_{\v{x}_k \in \L_k (\v{x}_0, \dots, \v{x}_{k-1})}
        \frac{1}{ |\L_{k+1} (\v{x}_0, \dots, \v{x}_{k}) |^2 }
\]
\[
    \cdot
    \sum_{\v{y} \in \L_{k+1} (\v{x}_0, \dots, \v{x}_{k})}
    g(\L_{k+1}^* (\v{x}_0, \dots, \v{x}_{k}), \v{y}) \,.
\]
Write the sum $S$ as $S' + S''$, where the summation in $S'$ is
taken over
$\v{x}_k \in  \t{B}_k (\v{x}_0, \dots, \v{x}_{k-1})$ and the
summation in $S''$ is taken over $\v{x}_k \in \L_k (\v{x}_0,
\dots, \v{x}_{k-1}) \setminus \t{B}_k (\v{x}_0, \dots,
\v{x}_{k-1})$.
Note that if $\v{x}_k \in \t{B}_k (\v{x}_0, \dots, \v{x}_{k-1})$, then the Bohr set \\
$\L_{k+1}^* (\v{x}_0, \dots, \v{x}_{k})$ is constructed by
Corollary  \ref{cor:together}. Using this corollary, we obtain
\begin{equation}\label{tmp:12:52_1}
  S' \ge
    \frac{1}{ |\L_k (\v{x}_0, \dots, \v{x}_{k-1}) |^2 }
        \sum_{\v{y} \in  \t{B}_k (\v{x}_0, \dots, \v{x}_{k-1}) }
            g(\L_k^* (\v{x}_0, \dots, \v{x}_{k-1}), \v{y}) ( 1 + \frac{\sigma^2}{32}) \,.
\end{equation}
Let us estimate the sum $S''$. Using Lemma \ref{l:l4}, we get
\begin{equation}\label{tmp:12:52_2}
  S'' \ge
        \frac{1}{ |\L_k (\v{x}_0, \dots, \v{x}_{k-1}) |^2 }
            \sum_{\v{y} \in \L_k (\v{x}_0, \dots, \v{x}_{k-1}) \setminus \t{B}_k (\v{x}_0, \dots, \v{x}_{k-1})}
                 g(\L_k^* (\v{x}_0, \dots, \v{x}_{k-1}), \v{y})  - 8 \k
\end{equation}
Combining (\ref{c:g_3'}), (\ref{tmp:12:52_1}), (\ref{tmp:12:52_2})
and (\ref{tmp:12:51'''}), we have
\[
  S \ge
    \frac{1}{ |\L_k (\v{x}_0, \dots, \v{x}_{k-1}) |^2 }
        \sum_{\v{y} \in \L_k (\v{x}_0, \dots, \v{x}_{k-1})} g(\L^*_k (\v{x}_0, \dots, \v{x}_{k-1}), \v{y})
        +
\]
\[
        +
    \frac{1}{ |\L_k (\v{x}_0, \dots, \v{x}_{k-1}) |^2 }
    \sum_{\v{y} \in \t{B}_k (\v{x}_0, \dots, \v{x}_{k-1}) }  2^{-13} \tau^2 \beta^2 \sigma^2   - 2^4 \k
    \ge
\]
\[
    \ge
     \frac{1}{ |\L_k (\v{x}_0, \dots, \v{x}_{k-1}) |^2 }
        \sum_{\v{y} \in \L_k (\v{x}_0, \dots, \v{x}_{k-1})} g(\L^*_k (\v{x}_0, \dots, \v{x}_{k-1}), \v{y})
        +
        2^{-14} \tau^2 \beta^2 \sigma^3 - 2^4 \k
\]
Using Lemma  \ref{l:l4}, we obtain
\[
  S \ge
        g ( \L_{k-1}^* (\v{x}_0, \dots, \v{x}_{k-2}), \v{x}_{k-1} ) + 2^{-14} \tau^2 \beta^2 \sigma^3 - 2^5 \k
    \ge
\]
\[
    \ge
        g ( \L_{k-1}^* (\v{x}_0, \dots, \v{x}_{k-2}), \v{x}_{k-1} ) + 2^{-15} \tau^2 \beta^2 \sigma^3
    \ge
\]
\begin{equation}\label{tmp:12:57}
    \ge
        g ( \L_{k-1}^* (\v{x}_0, \dots, \v{x}_{k-2}), \v{x}_{k-1} ) ( 1 + 2^{-15} \tau^2 \beta^2 \sigma^3 ) \,.
\end{equation}
On the other hand, $S_0$ is an estimate for $S$. Using Lemma
\ref{l:l4}, we get
\[
  S \ge S_0 - 8 \k \,.
\]
Thus if $\v{x}_{k-1}$ belongs to  $\o{R}_{k-1} (\v{x}_0, \dots,
\v{x}_{k-2})$ and $\v{x}_{k-1}$ satisfies (\ref{c:g_3}), then we
have
\begin{equation}\label{tmp:12:57`}
  S \ge g ( \L_{k-1}^* (\v{x}_0, \dots, \v{x}_{k-2}), \v{x}_{k-1} ) ( 1 + 2^{-15} \tau^2 \beta^2 \sigma^3 ) - 8 \k \,.
\end{equation}
Now suppose that  $\v{x}_{k-1}$ is an arbitrary vector,
$\v{x}_{k-1} \in \L_{k-1} (\v{x}_0, \dots, \v{x}_{k-2})$. Using
Lemma \ref{l:l4} twice, we have
\begin{equation}\label{tmp:12:57_7}
  S \ge g ( \L_{k-1}^* (\v{x}_0, \dots, \v{x}_{k-2}), \v{x}_{k-1} ) - 16 \k \,.
\end{equation}
\\

Let us consider $\ind_{k+1} (\mathbf{\L}_{k+1}) (g)$. We have
\[
  \ind_{k+1} (\mathbf{\L}_{k+1}) (g) =
\]
\[
               \frac{1}{|\L_0|^2} \sum_{\v{x}_0 \in \L_0} \frac{1}{|\L_1(\v{x}_0)|^2} \sum_{\v{x}_1 \in \L_1 (\v{x}_0)}
               \dots
               \sum_{\v{x}_{k-1} \in \L_{k-1} (\v{x}_0, \dots, \v{x}_{k-2}) } S(\v{x}_0, \dots, \v{x}_{k-1}) \,.
\]
By assumption $\ind_{k-1} (\mathbf{\L}_{k-1}, \o{R}_{k-1} ) (g_3)
\ge \tau \beta /4$. In other words
\[
               \frac{1}{|\L_0|^2} \sum_{\v{x}_0 \in \L_0} \frac{1}{|\L_1(\v{x}_0)|^2} \sum_{\v{x}_1 \in \L_1 (\v{x}_0)}
               \dots
\]
\begin{equation}\label{f:tmp_22:17_7}
               \sum_{\v{x}_{k-1} \in \o{R}_{k-1} (\v{x}_0, \dots, \v{x}_{k-2}) }
                    g_3 ( \L_{k-1}^* (\v{x}_0, \dots, \v{x}_{k-2}), \v{x}_{k-1} )
                        \ge
                            \tau \beta /4 \,.
\end{equation}
By $M_{k-1} (\v{x}_0, \dots, \v{x}_{k-2}) $ denote the set of
$\v{x}_{k-1} \in \o{R}_{k-1} (\v{x}_0, \dots, \v{x}_{k-2})$ such
that $\v{x}_{k-1}$ satisfies (\ref{c:g_3}). Using
(\ref{f:tmp_22:17_7}), we obtain
\[
    S_M :=
               \frac{1}{|\L_0|^2} \sum_{\v{x}_0 \in \L_0} \frac{1}{|\L_1(\v{x}_0)|^2} \sum_{\v{x}_1 \in \L_1 (\v{x}_0)}
               \dots
\]
\begin{equation}\label{f:tmp_22:24_7}
               \sum_{\v{x}_{k-1} \in M_{k-1} (\v{x}_0, \dots, \v{x}_{k-2}) }
                    g_3 ( \L_{k-1}^* (\v{x}_0, \dots, \v{x}_{k-2}), \v{x}_{k-1} )
                        \ge
                            \tau \beta /8 \,.
\end{equation}
Using (\ref{c:g_3}), (\ref{tmp:12:57`}), (\ref{tmp:12:57_7}) and
(\ref{f:tmp_22:24_7}), we get
\[
  \ind_{k+1} (\mathbf{\L}_{k+1} ) (g)
                \ge
                \frac{1}{|\L_0|^2} \sum_{\v{x}_0 \in \L_0} \frac{1}{|\L_1(\v{x}_0)|^2} \sum_{\v{x}_1 \in \L_1 (\v{x}_0)}
                \dots
\]
\[
                \Big\{
                \sum_{\v{x}_{k-1} \in M_{k-1} (\v{x}_0, \dots, \v{x}_{k-2}) }
                (  g ( \L_{k-1}^* (\v{x}_0, \dots, \v{x}_{k-2}), \v{x}_{k-1} ) ( 1 + 2^{-15} \tau^2 \beta^2 \sigma^3 ) - 8 \k)
                    +
\]
\[
                    +
                \sum_{\v{x}_{k-1} \in \L_{k-1} (\v{x}_0, \dots, \v{x}_{k-2} ) \setminus M_{k-1} (\v{x}_0, \dots, \v{x}_{k-2})}
                (  g ( \L_{k-1}^* (\v{x}_0, \dots, \v{x}_{k-2}), \v{x}_{k-1} ) - 16 \k) \Big\}
             \ge
\]
\[
             \ge
                \ind_{k-1} (\mathbf{\L}_{k-1}) (g)
                +
                  2^{-15} \tau^2 \beta^2 \sigma^3 \Big( \frac{\tau \beta}{8} \Big)
                  S_M
                  - 24 \k
             \ge
\]
\[
                \ge
                    \ind_{k-1} (\mathbf{\L}_{k-1} ) (g) + 2^{-24} \tau^4 \beta^4 \sigma^3  - 24 \k
                \ge
\]
\[
                \ge
                    \ind_{k-1} (\mathbf{\L}_{k-1} ) (g) + 2^{-25} \tau^4 \beta^4 \sigma^3 \,.
\]
In other words, for all $k\ge 1$, we have
\begin{equation}\label{e:inc_ind}
  \ind_{k+1} (\mathbf{\L}_{k+1}) (g) \ge \ind_{k-1} (\mathbf{\L}_{k-1}) (g)
                                                            + 2^{-25} \tau^4 \beta^4 \sigma^3 \,.
\end{equation}

Since for any $k$ we have $\ind_k (\mathbf{\L}_k) (g) \le 1$, it
follows that the total number of steps of the algorithm does not
exceed $K_0=2^{30} \tau^{-4} \beta^{-4} \sigma^{-3}$.

Suppose that the algorithm stops at step $K$, $K\ge 1$, $K \le
2^{30} \tau^{-4} \beta^{-4} \sigma^{-3}$. We have
\begin{equation}\label{eq:g_e}
  \ind_{K} (\mathbf{\L}_K, \o{R}_K) (g_3)   <  \frac{\tau \beta}{4} \,.
\end{equation}
Using Lemma \ref{l:keps}, we get
\[
  \ind_{K} (\mathbf{\L}_K) (g_1) \ge (\d + \tau) \beta - 8 \k K  \ge (\d +  \frac{7 \tau}{8} ) \beta \,.
\]
Using (\ref{eq:g_e}), we obtain
\begin{equation}\label{e:t_20:05}
  \ind_{K} ( \mathbf{\L}_K, R_K ) (g_1) \ge (\d +  \frac{3 \tau}{8} ) \beta \,.
\end{equation}
The summation in (\ref{e:t_20:05}) is taken over the sets $\L_K^*
(\v{x}_0, \dots, \v{x}_{K-1}) + \v{y}$, where $\v{y} \in R_K
(\v{x}_0, \dots, \v{x}_{K-1}) $.

Let $E_K$ be the family of vectors $\v{y}$ such that
$\v{y} \in E_{K} (\v{x}_0, \dots, \v{x}_{K-1})$, and $R_K^*$ be
the family of vectors  $\v{y}$ such that
$\v{y} \in R_{K} (\v{x}_0, \dots, \v{x}_{K-1})$, but $\v{y}$ does
not belong to  $E_{K} (\v{x}_0, \dots, \v{x}_{K-1})$. We have
\begin{equation}\label{tmp:19:47}
  \ind_K (\mathbf{\L}_K, E_K) (g_1)
            < \frac{\tau \beta}{16} \ind_K (\mathbf{\L}_K) (1) \le \frac{\tau \beta}{16}  \,.
\end{equation}
Combining (\ref{e:t_20:05}), (\ref{tmp:19:47}), we get
\begin{equation}\label{tmp:20:25_t}
  \ind_K (\mathbf{\L}_K, R_K^*) (g_1) > (\d + \frac{\tau}{4} ) \beta \,.
\end{equation}
Suppose that for  all $\v{y} \in R_K^* (\v{x}_0, \dots,
\v{x}_{K-1})$, we have $g_2(\L_K^* (\v{x}_0, \dots,
\v{x}_{K-1}),\v{y}) < (\d + \tau/16 )$. Then
\[
  (\d + \frac{\tau}{4} ) \beta  < \ind_K (\mathbf{\L}_K, R_K^*) (g_1)
                        \le (\d + \frac{\tau}{16} ) \ind_{K} (\mathbf{\L}_K, R_K^*) (g_3)
                        \le
\]
\begin{equation}\label{}
                        \le (\d + \frac{\tau}{16} ) \ind_{K} (\mathbf{\L}_K) (g_3) \,.
\end{equation}
Using Lemma \ref{l:keps} once again, we obtain
\[
  (\d + \frac{\tau}{4} ) \beta  < (\d + \frac{\tau}{16} ) \ind_{K} (\mathbf{\L}_K) (g_3)
                        \le (\d + \frac{\tau}{16} ) (\beta + 8 \k K) \le (\d + \frac{\tau}{4} ) \beta
\]
with contradiction. Whence there exist vectors
$\v{x}_0, \dots, \v{x}_{K-1}$, $\v{y}$ such that \\
$
  g_2(\L_K^* (\v{x}_0, \dots, \v{x}_{K-1}),\v{y}) \ge (\d + \tau/16 )
$ and $\v{y} \in R_{K} (\v{x}_0, \dots, \v{x}_{K-1}) \setminus
E_{K} (\v{x}_0, \dots, \v{x}_{K-1})$. Put $\v{t} = \v{y} + \v{s}$
and $\L' = \L_K^* (\v{x}_0, \dots, \v{x}_{K-1})$. We obtain the
vector $\v{t}$, the sets $E_1' = (\t{E}_1 - y_1) \cap \L'$, $E_2'
= (\t{E}_2 - y_2) \cap \L'$ and the Bohr set $\L'$ which satisfy
the conditions $1)$---$3)$.

Let us estimate $D$ and $\eps'$. At the each step of the algorithm
the dimension of Bohr sets increases at most $1$. Since the total
number of steps does not exceed $K_0$, it follows that $D \le d +
2^{30} \tau^{-5} \beta^{-5} \sigma^{-3}$ and $\eps' \ge (2^{-10}
\eps)^D \eps_0$. Using Lemma \ref{l:Bohr_est} and (\ref{c:Nge}),
we obtain that the set  $\L'$ is not empty.
This completes the proof.
\end{proof*}


\section{Proof of main result.}

Let us put Theorem \ref{t:Phase1} and Proposition
\ref{prop:point_3}  together in a single proposition.

\begin{proposition}
  Let $\L= \L(S, \eps_0)$ be a Bohr set, $|S| = d$,
  and $\v{s} = (s_1,s_2) \in G\m G$.
  Let   $E_1$, $E_2$ be sets, $E_i = \beta_i |\L|$, $i=1,2$, $\beta = \beta_1 \beta_2$.
  Suppose
  ${\bf E} = E_1 \m E_2$ is a subset of $(\L+s_1) \m (\L+s_2)$,
  $E_1$, $E_2$ are $(\a_0,2^{-10} \eps^2)$--uniform subsets of $\L + s_1$, $\L + s_2$, respectively,
  $\a_0 = 2^{-2000} \d^{96} \beta_1^{48} \beta_2^{48}$, $\eps = (2^{-100}  \a_0^2 ) / (100 d)$.
  Suppose that $A$ is a subset of ${\bf E}$, $\d_{{\bf E}} (A) = \d$,
  and $A$ has no triples $\{ (k,m), (k+d,m), (k,m+d) \}$ with $d\neq 0$.
  Let
\begin{equation}\label{END!}
    \log N \ge 2^{1000 000} (2^{250 000} \d^{-20 000} \beta^{-200} + d )^3 \log \frac{1}{\d \beta \eps_0 }  \,.
\end{equation}
    Then there is a Bohr set $\t{\L}$
    and a vector $\v{y} = (y_1, y_2) \in G\m G$ with the following properties :
    there exist sets $E_1' \subseteq ( E_1  - y_1 \cap \t{\L} ) $, $E_2' \subseteq ( E_2 - y_2 \cap \t{\L} ) $
    such that
\\
\\
$
  \quad 1) ~~
        \mbox{ Let }
        |E_1'| = \beta_1' |\t{\L}|, |E_2'| = \beta_2' |\t{\L}|
                        \mbox{ and } \beta' = \beta_1' \beta_2' \,.
        \mbox{ Then }
         \beta'  \ge 2^{-1500} \d^{100} \beta \,.
$
\\
$
  \quad 2) ~~  E_1', E_2' \mbox{ are } (\a_0', 2^{-10} \eps'^2)\mbox{--uniform, where }
        \a_0' = 2^{-2000} \d^{96} \beta'^{48},
$
\\
$
         \eps' = \frac{ 2^{-100} \a_0'^2}{ 100 D' } \,,
         D\le D' = 2^{250 000} \d^{-20000} \beta^{-200} + d \,.
%
%
$
\\
$
   \quad 3) ~~ \mbox{ For } \t{\L} = \L (\t{S},\t{\eps})  \mbox{ we have }
$ $
   |\t{S}| = D ,
$ $
    \mbox{ and }
   \t{\eps}
          \ge ( 2^{-100} \eps'^2 )^D  \eps_0 \,.
$
\\
$
  \quad 4) ~~
  \d_{E_1' \m E_2' } (A) \ge \d + 2^{-600} \d^{22} \,.
$ \label{pred:Ph1_Ph2}
\end{proposition}

{\bf Proof of Theorem \ref{main_th}}. Suppose that $A \subseteq G
\m G$, $|A| = \d N^2$ and $A$ has no  triples $\{ (k,m), (k+d,m),
(k,m+d) \}$ with $d\neq 0$.

The proof of Theorem \ref{main_th} is a sort of an algorithm.

After the $i$th step of the algorithm a vector  $\v{s}_i =
(s^{(1)}_i,s^{(2)}_i)$ and sets : a regular Bohr set $\L_i = \L
(S_i, \eps_i)$, sets $E_i^{(1)} - s^{(1)}_i \subseteq \L_i $,
$E_i^{(2)} - s^{(2)}_i \subseteq \L_i $, will be constructed. Let
$|E_i^{(1)}| = \beta^{(1)}_i |\L_i|$,  $|E_i^{(2)}| =
\beta^{(2)}_i |\L_i|$, $\beta_i = \beta^{(1)}_i \beta^{(2)}_i$,
${\bf E_i} = E_i^{(1)} \m E_i^{(2)}$.

The sets $\L_i$, $E_i^{(1)}$, $E_i^{(2)}$ satisfy the following
conditions
\\
$ 1)~~~ \beta_{i}  \ge 2^{-1500} \d^{100} \beta_{i-1} $. \\
$ 2)~~~ E_i^{(1)}, E_i^{(2)}$ are $(\a_0^{(i)}, 2^{-10}
(\eps'_{i})^2)$--uniform,
$ \a_0^{(i)} = 2^{-2000} \d^{96} \beta_i^{48}$, $\eps'_{i} = 2^{-100} (\a_0^{(i)})^2 / (100 d_i )$. \\
$ 3)~~~ \L_i = \L (S_i, \eps_i),
  |S_i| = {d_i} , d_i \le 2^{250 000} \d^{-20000} \beta_{i-1}^{-200} + d_{i-1} ,
$
$
   \eps_{i} \ge (2^{-100} (\eps'_{i})^2 )^{d_i} \eps_{i-1}
$.
\\
$ 4)~~~ \d_{\bf E_i} (A') \ge \d_{\bf E_{i-1} } (A') + 2^{-600}
\d^{22} $.
\\

Proposition \ref{pred:Ph1_Ph2} allows us to carry the $(i+1)$th
step of the algorithm. By this Proposition there exists a new
vector $\v{s}_{i+1} = (s^{(1)}_{i+1},s^{(2)}_{i+1}) \in G\m G$ and
sets : a regular Bohr set $\L_{i+1} = \L (S_{i+1}, \eps_{i+1})$,
sets $E_{i+1}^{(1)} - s^{(1)}_i \subseteq \L_{i+1} $,
$E_{i+1}^{(2)} - s^{(2)}_{i+1} \subseteq \L_{i+1} $, ${\bf
E_{i+1}} = E_{i+1}^{(1)} \m E_{i+1}^{(2)}$, which satisfy $1)$ ---
$4)$.

Put $S_0 = \{ 0 \} $, $\L_0 = \L (S_0, 1)$ and $E_1 = E_2 = G$,
$\beta_0 = 1$. Clearly, $E_1$, $ E_2$  are $(2^{-2000} \d^{96},
2^{-10 000}  \d^{400} )$--uniform. Hence we have constructed
zeroth step of the algorithm.

Let us estimate the total number of steps of our procedure. For an
arbitrary $i$ we have $ \d_{\bf E_i} (A') \le 1$. Using this and
condition $4)$, we obtain that the total number of steps cannot be
more then $ 2^{700} \d^{-21} = K $.

Condition $3)$ implies  $\beta_i \ge (2^{-1500} \d^{100})^i$.
Hence  $d_i \le (C_1 \d)^{ -C'_1 i }$, where $C_1, C'_1 > 0$ are
absolute constants.

To prove  Theorem \ref{main_th}, we                                                                 
need to verify
condition (\ref{END!}) at the last step of the algorithm.
Condition (\ref{END!}) can be rewrite as
\begin{equation}\label{END!!}
  N \ge (C'_2 \d )^{-C'_3 \d^{-C'_4 K}} =  \exp ( \d^{ - C' \d^{-21} }) \,,
\end{equation}
where $C'_2, C'_3, C'_4, C' > 0$ are absolute constants. By
assumption
\[
  \d \gg \frac{1}{ ( \log \log N )^{1/22} }
\]
and we get (\ref{END!!}). Hence  $A'$ has a triple $\{ (k,m),
(k+d,m), (k,m+d) \}$, where $d\neq 0$. This contradiction
concludes the proof.

\begin{note*} Certainly,  the constant $14$ in Theorem  \ref{main_th}  can be slightly decreased.
Nevertheless, it is the author's opinion that this constant cannot
be lowered  to anything like $1$ without a new idea.
\end{note*}

Using the following lemma of B. Green (see e.g.
\cite{Shkr_tri_LMS} or \cite{Gow_hyp4}) one can obtain a corollary
of Theorem \ref{main_th} concerning subsets of $\{ -N,\dots,N
\}^2$ without corners (see details in \cite{Shkr_tri_LMS}).

\begin{lemma}
  Let  $N$ be a natural number.
  Suppose  $A$ is a subset of $\{ -N,\dots,N \}^2$, $|A| = \d (2N+1)^2$, and
  $A$ has no triples $\{ (k,m), (k+d,m), (k,m+d) \}$ with $d>0$.
  Then there exists a set $A_1 \subseteq A$ such that \\
  $1)~$
  $
   |A_1| \ge \d^2  (2N+1)^2 /4
  $
  and \\
  $2)~$ $ A_1 $ has no triples  $ \{ (k,m), (k+d,m), (k,m+d) \}$ with $d\neq 0$.
\label{l:Green's_trick}
\end{lemma}

\begin{corollary}
 Let $\delta>0$, and $N\gg \exp \exp ( \d^{-43} )$.
 Let  $A$ be a subset of
 $\{1,\dots,N\}^2$ of cardinality at least  $\delta N^2$.
 Then $A$ contains a triple
 $\{ (k,m), (k+d,m), (k,m+d) \}$ with $d>0$.
\end{corollary}



\affiliationone{
   I.D. Shkredov\\
   Division of Dynamical Systems\\
   Department of Mechanics and Mathematics\\
   Moscow State University\\
   Leninskie Gory, Moscow, 119991\\
   Russia
   \email{ishkredov@rambler.ru\\
   ishkredo@mech.math.msu.su}
   }

\end{document}